\documentclass[a4paper]{article}

\usepackage[utf8]{inputenc}
\usepackage[T1]{fontenc}
\usepackage{newtxmath}

\usepackage{amsmath}

\usepackage{amsthm}
\usepackage{bbm}
\usepackage{mathtools}
\usepackage{calrsfs}
\usepackage{tensor}  
\usepackage{dsfont}  
\DeclareMathAlphabet{\pazocal}{OMS}{zplm}{m}{n} 

\usepackage{fixmath} 

\usepackage[verbose=true,margin=1.1in]{geometry} 
\usepackage{parskip}

\usepackage[usenames,dvipsnames,svgnames,table]{xcolor} 
\usepackage[inline,shortlabels]{enumitem}
\usepackage{relsize} 

\usepackage{hyperref}
\hypersetup{
  bookmarksnumbered,
  colorlinks=true,
  linkcolor=DarkGreen,
  citecolor=Purple,
  urlcolor=Purple
}
\usepackage[capitalise]{cleveref}

\crefmultiformat{defn}{Definitions~#2#1#3}{ and #2#1#3}{, #2#1#3}{ and~#2#1#3}

\usepackage{tikz-cd}
\usepackage{comment}

\usepackage{todonotes}

\usepackage[inline,shortlabels]{enumitem}
\usepackage{etoolbox}

\setlist[enumerate]{nosep,label=(\arabic*),font=\normalfont,leftmargin=2.8em}
\AtBeginEnvironment{enumerate}{~}

\newcommand*{\wipp}{\noindent\bgroup\color{Purple}} 

\usepackage{amsmath}
\usepackage{amsthm}
\usepackage{mathrsfs}
\usepackage{fixmath}
\usepackage{upgreek}
\usepackage{xparse} 
\usepackage{mathtools} 

\makeatletter
\theoremstyle{plain}
\newtheorem{thm}{Theorem}[section] 
\newtheorem{prop}[thm]{Proposition}
\newtheorem{lem}[thm]{Lemma}
\newtheorem{cor}[thm]{Corollary}

\theoremstyle{definition}
\newtheorem{defn}[thm]{Definition}
\newtheorem{rem}[thm]{Remark}
\newtheorem{example}[thm]{Example}

\theoremstyle{plain}

\newcommand{\ncmd}[1]{\newcommand{#1}}
\newcommand{\rncmd}[1]{\renewcommand{#1}}
\ncmd{\opname}[1]{\operatorname{#1}}

\ncmd{\mrm}[1]{\mathrm{#1}}
\ncmd{\mcal}[1]{\mathcal{#1}}
\ncmd{\pcal}[1]{\pazocal{#1}}
\ncmd{\mfr}[1]{\mathfrak{#1}}
\ncmd{\mbm}[1]{\mathbbm{#1}}
\ncmd{\mbb}[1]{\mathbb{#1}}
\ncmd{\mbf}[1]{\mathbf{#1}}

\ncmd{\tild}[1]{\widetilde{#1}}

\ncmd{\Scomment}[1]{{\color{teal}{\bf [#1]}}}

\ncmd{\donetodo}[1]{\todo[color=green!50]{\checkmark\quad #1}}
\ncmd{\hl}[1]{\textcolor{purple}{#1}}

\ncmd{\xto}[1]{\xrightarrow{#1}}
\ncmd{\mono}{\hookrightarrow}
\ncmd{\epi}{\twoheadrightarrow}
\ncmd{\iso}{\xrightarrow{\,\smash{\raisebox{-0.5ex}{\ensuremath{\scriptstyle\sim}}}\,}}
\ncmd{\adj}{\rightleftarrows}
\ncmd{\too}{\longrightarrow}

\ncmd{\categories}{$\infty$-categories}
\ncmd{\category}{$\infty$-category}
\ncmd{\forget}{\text{forget}}

\rncmd{\amalg}{\sqcup}
\newcommand{\colim}{\operatornamewithlimits{colim}}
\ncmd{\xlim}[1]{\underset{#1}{\lim}}
\ncmd{\heart}{\heartsuit}

\ncmd{\BAut}{\mrm{BAut}}
\ncmd{\Aut}{\mrm{Aut}}
\ncmd{\unit}{\mathds{1}}
\ncmd{\Psh}{\mrm{Psh}}
\ncmd{\Cat}{\mathrm{Cat}}
\ncmd{\Ar}{\mrm{Ar}}
\ncmd{\Sect}{\mrm{Sect}}
\ncmd{\fib}{\mrm{fib}}
\ncmd{\pt}{\mathrm{pt}}
\ncmd{\Fun}{\mathrm{Fun}}
\ncmd{\op}{\mathrm{op}}
\ncmd{\Id}{\mathrm{Id}}
\ncmd{\can}{\mrm{can}}
\ncmd{\Map}{\mathrm{Map}}
\ncmd{\map}{\mathrm{map}}
\ncmd{\PrL}{{\mathbf{Pr}^{\mathrm{L}}}}
\ncmd{\PrLst}{\mathbf{Pr}^{\mathrm{L}}_{\st}}
\ncmd{\End}{\mrm{End}}
\ncmd{\Ind}{\mrm{Ind}}
\ncmd{\Sp}{\mrm{Sp}}
\ncmd{\st}{\mrm{st}}

\ncmd{\conn}{\mrm{conn}}
\ncmd{\trun}{truncated}
\ncmd{\cof}{\mrm{cof}}
\ncmd{\Cell}{\mathsf{C}}
\ncmd{\Conn}{\widehat{\mathsf{C}}}
\ncmd{\Null}{\mrm{Null}}
\ncmd{\Conull}{\mrm{Conull}}
\ncmd{\type}{\mrm{type}}
\ncmd{\parL}{\pcal{L}^f}

\ncmd{\Alg}{\mrm{Alg}}
\ncmd{\Mon}{\mrm{Mon}}
\ncmd{\Mod}{\mrm{Mod}}
\ncmd{\Free}{\mrm{Free}}
\ncmd{\HFp}{H\mbb{F}_p}
\ncmd{\HQ}{H\mbb{Q}}
\ncmd{\Fp}{\mbb{F}_p}
\ncmd{\coker}{\opname{coker}}

\ncmd{\calC}{\mcal{C}}
\ncmd{\calB}{\pcal{B}}
\ncmd{\calF}{\pcal{F}}
\ncmd{\calL}{\pcal{L}}
\ncmd{\bbZ}{\mbm{Z}}
\ncmd{\bbS}{\mbm{S}}
\ncmd{\calA}{\mcal{A}}
\ncmd{\bfC}{\mbf{C}}
\ncmd{\bfP}{\mbf{P}}
\ncmd{\bbQ}{\mbm{Q}}
\ncmd{\bfL}{\mbf{L}}
\ncmd{\bbE}{\mbb{E}}
\ncmd{\calS}{\pcal{S}}
\ncmd{\calK}{\pcal{K}}
\ncmd{\calU}{\pcal{U}}
\ncmd{\frM}{\mfr{M}}
\ncmd{\frZ}{\mfr{Z}}

\DeclareMathOperator{\Loc}{L} 

\newcommand*{\Loop}{\operatorname{\Omega}\mathopen{}}
\newcommand*{\susp}{\operatorname{\Sigma}\mathopen{}}

\ncmd{\cell}[1]{\langle #1\rangle}

\ExplSyntaxOn
\NewDocumentCommand{\mathcalUpper}{m}{
  \tl_set:Nn \l_tmpa_tl { #1 }
  \regex_replace_all:nnN { ([A-Z]+) } { \c{mathcal}\cB\{\1\cE\} } \l_tmpa_tl
  \tl_use:N \l_tmpa_tl
  }
\ExplSyntaxOff

\usepackage{xspace}
\usepackage{xpunctuate}

\let\oldiff\iff
\renewcommand*{\iff}{\ifmmode\oldiff\else{if and only if}\xspace\fi}

\usepackage[style=alphabetic, backend=biber, maxbibnames=99,
  minbibnames=99, doi=false, isbn=false, url=false, giveninits = true]{biblatex}
\title{On periodic homotopy and homology equivalences of spaces}
\author{Shaul Barkan, Gijs Heuts and Yuqing Shi}

\begin{document}

  \maketitle

\begin{abstract}
There are at least two ways to approach the homotopy theory of spaces `at chromatic height $n$': one may localize with respect to $T(n)$-homology or with respect to $v_n$-periodic homotopy groups. It was already observed by Bousfield that these two options yield rather different results. We build on his work to prove precise comparison results between the two notions. A crucial concept is a more robust notion of $T(n)$-equivalence that we call `parametric $T(n)$-equivalence': this is a map of spaces that induces an equivalence on $\infty$-categories of local systems valued in $T(n)$-local spectra. Our results are sharpest in the case of infinite loop spaces, where amongst other things we prove a $T(n)$-local version of a result of Kuhn on the Morava $K$-theory of the Whitehead tower. As a corollary of our results we also produce a formula for the $\Loc_n^f$-localization of an infinite loop space $\Loop^\infty E$ of a spectrum satisfying $\Loc_{n-1}^f E \simeq 0$.
\end{abstract} 
  
  \tableofcontents

  \section{Introduction}

The chromatic philosophy is the central organizing principle of stable homotopy theory. One studies the $\infty$-category of spectra by decomposing it into `monochromatic' pieces. This can either mean localizing with respect to the prime fields $K(n)$, the Morava $K$-theories, or localizing with respect to the telescopes $T(n)$, recording the $v_n$-periodic homotopy groups of spectra. The two options are closely related, but have recently been shown to differ when $n \geq 2$ \cite{telescopeconj}.

In this paper we are concerned with unstable homotopy theory,~i.e., the homotopy theory of spaces. As in the stable case one can decompose the $\infty$-category of spaces into \emph{periodic} or \emph{monochromatic} pieces, but there is potential ambiguity in what one means by this precisely. There are at least the following two possibilities to consider:
\begin{itemize}
\item[(1)] The localization of the $\infty$-category of pointed spaces $\mathcal{S}_*$ with respect to $v_n$-periodic homotopy groups, denoted $\mathcal{S}_{v_n}$. We will review the $\infty$-category $\mathcal{S}_{v_n}$ in Section \ref{sec:periodichomotopy}. This localization can be viewed as a generalization of rational homotopy theory to heights $n > 0$. The $\infty$-category $\mathcal{S}_{v_n}$ admits a description in terms of Lie algebras analogous to Quillen's results on rational homotopy theory \cite{heutslie}.
\item[(2)] The localization of the $\infty$-category of spaces $\mathcal{S}$ with respect to $T(n)$-homology isomorphisms (or $K(n)$-homology isomorphisms). This is an example of Bousfield localization with respect to a homology theory \cite{Bou75}.
\end{itemize}
The first perspective emphasizes homotopy groups, the second focuses on homology. For the $\infty$-category of spectra, $v_n$-periodic homotopy and $T(n)$-homology groups are the same thing, so the distinction above does not play a role in stable homotopy theory.
Also, in the case $n=0$ the theory of Serre classes shows that for spaces rational homotopy equivalences and rational homology isomorphisms agree, at least in the simply-connected case, so that (1) and (2) are almost identical. However, in the cases $n>0$, the localizations (1) and (2) behave \emph{very} differently for the category of spaces. Let us illustrate this by some examples:
\begin{itemize}
\item[(a)] A $v_n$-periodic equivalence of pointed spaces $f\colon X \to Y$ need not be a $T(n)$-homology isomorphism. For example, any truncated space $X = \tau_{\leq k} X$ has vanishing $v_n$-periodic homotopy groups, but need not have vanishing $T(n)$-homology. Indeed, the $T(n)$-homology of Eilenberg--MacLane spaces $K(\mathbb{F}_p,k)$ is nontrivial for $k \leq n$ by the calculations of Ravenel--Wilson~\cite{RW80}.
\item[(b)] There exist more sophisticated examples of $v_n$-periodic equivalences that are not $T(n)$-homology isomorphisms and which cannot be avoided by imposing connectivity hypotheses. Indeed, consider the space $X = BU$. Then $X$ has vanishing $v_2$-periodic homotopy groups, but it does not have vanishing $T(2)$-homology. One can show that any highly connected cover of $BU$ will have the same features. (See \Cref{rmk:multipleheights} for more on this phenomenon.)
\item[(c)] In the converse direction, a $T(n)$-homology isomorphism need not be a $v_n$-periodic equivalence. For example, Langsetmo--Stanley~\cite{LS01} construct modifications of the Adams map $\alpha\colon S^{k+ 2(p-1)}/p \to S^k/p$ that retain the property of it being a $T(1)$-homology isomorphism, but fail to be $v_1$-periodic equivalences.
\end{itemize}

In his extensive work on homotopical and homological localizations, Bousfield gives many results on the relation between the notion of $v_n$-periodic equivalences and $T(n)$-homology isomorphisms \cite{Bou94,Bou97,Bou01}. These play an important role in recent developments in unstable periodic homotopy theory \cite{heutslie} and in the study of telescopically localized algebraic $K$-theory \cite{LMMT}. In this paper we further clarify the relation between localizations (1) and (2) above. 

\subsubsection{Main results}

The starting point of our work is the following: rather than focusing on maps $f$ inducing isomorphisms on $T(n)$-homology, it turns out to be more robust to consider maps giving $T(n)$-isomorphisms `with local coefficients', in the following sense. We call a map $f\colon X \to Y$ of spaces a \emph{parametric $T(n)$-equivalence} if it induces an equivalence of $\infty$-categories
\[
f^*\colon \Fun(Y,\Sp_{T(n)}) \to \Fun(X,\Sp_{T(n)}).
\]
In other words, we want $f$ to identify the $\infty$-categories of parametrized $T(n)$-local spectra over $X$ and over $Y$, hence the term. A parametric $T(n)$-equivalence is in particular a $T(n)$-homology isomorphism; we will discuss the relation between the two notions in more detail in \Cref{sec:homologyequiv}.

The main thread of this paper will be to compare the three notions of $v_n$-periodic equivalence, $T(n)$-equivalence, and parametric $T(n)$-equivalence. Theorems \ref{thm:mainabs}, \ref{thm:mainhomologytohtpy}, and \ref{thm:mainhtpytohomology} are general comparison results that are sharpenings of Bousfield's existing work on the topic. The main novelty here is \Cref{thm:maininfloop}, which gives a rather sharp answer for infinite loop spaces. Its proof involves not only our first set of results, but also a collection of results on $T(n)$-local $\mathbb{E}_\infty$-rings that might be of independent interest.

Note that when we speak of $T(n)$, we have implicitly already chosen a prime number $p$. To state our results we adopt the following somewhat nonstandard convention: we say a space $X$ is \emph{weakly $p$-local} if for any basepoint $x \in X$, the homotopy groups $\pi_n(X,x)$ are $p$-local abelian groups for $n \geq 2$. (We are \emph{not} assuming anything about $\pi_0$ and $\pi_1$, nor any type of nilpotence condition.) Similarly, we call $X$ \emph{weakly $p^\infty$-torsion} if the homotopy groups $\pi_n(X,x)$ are $p^\infty$-torsion (meaning every element is annihilated by some power of $p$) for $n \geq 2$. In other words, $X$ is weakly $p^\infty$-torsion if it is weakly $p$-local and its simply connected cover has trivial rationalization.

\begin{thm}
\label{thm:mainabs}
Suppose $X$ is a weakly $p$-local space and $n \geq 0$ a natural number. Then $X \to *$ is a parametric $T(i)$-equivalence for every $i \leq n$ if and only if $X$ is $(n+1)$-connected and has vanishing $v_i$-periodic homotopy groups for every $i \leq n$.
\end{thm}

Of course it is desirable to upgrade \Cref{thm:mainabs} from an absolute statement about $X$ to a relative statement about a map $f\colon X \to Y$. First of all we build on Bousfield's work \cite{Bou94,Bou97} to prove the following two results:
 
\begin{thm}
\label{thm:mainhomologytohtpy}
Suppose $f\colon X \to Y$ is a parametric $T(n)$-equivalence of weakly $p^\infty$-torsion spaces. Then the truncation $\tau_{\leq n+1} f$ is an equivalence and $f$ is a $v_n$-periodic equivalence at every basepoint.
\end{thm}

For the next statement, recall that a map $f$ is said to be $k$-connected if its fiber (over any basepoint) is $(k-1)$-connected, or equivalently if $f$ induces isomorphisms on $\pi_i$ for $i < k$ and a surjection on $\pi_k$, taken at any basepoint.

\begin{thm}
\label{thm:mainhtpytohomology}
Suppose $f\colon X \to Y$ is an $(n+2)$-connected map of spaces and that $f$ is a $v_i$-periodic equivalence at every basepoint, for all $i \leq n$. Then $f$ is a parametric $T(i)$-equivalence for every $i \leq n$.
\end{thm}

\begin{rem}
\label{rmk:multipleheights}
One important difference between \Cref{thm:mainhomologytohtpy} and \Cref{thm:mainhtpytohomology} is that the first focuses on a single height $n$, while the second concerns all heights $i \leq n$. This is necessary. Indeed, the statement of \Cref{thm:mainhtpytohomology} \emph{cannot} be improved to a statement of the following kind:
\begin{itemize}
\item[(*)] An $(n+2)$-connected map $f\colon X \to Y$ of spaces that is a $v_n$-periodic equivalence at every basepoint is also a parametric $T(n)$-equivalence.
\end{itemize}
Indeed, there exist highly connected spaces $X$ with trivial $v_n$-periodic homotopy groups, but nontrivial $T(n)$-homology (see example (b) earlier in this introduction or \Cref{rmk:vnversusTn} below for details). In particular, for such $X$ the map $X \to *$ is also not a parametric $T(n)$-equivalence.
\end{rem}

\begin{rem}
\label{rem:Bousfieldcounterexample}
Note that \Cref{thm:mainhtpytohomology} offers a partial converse to \Cref{thm:mainhomologytohtpy}. However, the first assumes $f$ to be $(n+2)$-connected whereas the second gives a statement about $\tau_{\leq n+1} f$, which is slightly weaker. One might hope that this is just a limitation of our proof technique and optimistically propose a statement of the following kind, for a fixed natural number $n \geq 0$:
\begin{itemize}
\item[(**)] A map $f\colon X \to Y$ of weakly $p^\infty$-torsion spaces is a parametric $T(i)$-equivalence for every $i \leq n$ if and only if $\tau_{\leq n+1} f$ is an equivalence and $f$ is a $v_i$-periodic equivalence for every $i \leq n$.
\end{itemize}
The `only if' follows from \Cref{thm:mainhomologytohtpy} and the `if' would be a strengthening of \Cref{thm:mainhtpytohomology}. However, (**) \emph{does not hold} at this level of generality. Indeed, it follows from \cite[Lemma 8.8]{Bou01} and \Cref{lem:parametricvsordinary} that the map
\[
\tau_{> 4} \Omega^\infty \mathbb{S}^4/p \to \Omega^\infty\mathbb{S}^4/p
\]
is not a parametric $K(2)$-equivalence, hence not a parametric $T(2)$-equivalence. (Here $\mathbb{S}^4/p$ is the 4-fold suspension of the mod $p$ Moore spectrum.) This map is an equivalence on 3-truncations (since both spaces are 3-connected) and a $v_i$-periodic equivalence for any $i \geq 0$, contradicting statement (**) above in the case $n=2$. A different attempt to better align the statements of \Cref{thm:mainhomologytohtpy} and \Cref{thm:mainhtpytohomology} would be to ask if the conclusion of \Cref{thm:mainhomologytohtpy} can be strengthened to say that $f$ is in fact $(n+2)$-connected. This cannot work either, since $* \to K(\mathbb{F}_p,n+2)$ is a parametric $T(n)$-equivalence (by \Cref{prop:T(n)localEMspaces} and \Cref{lem:parametricvsordinary}) which is not $(n+2)$-connected.
\end{rem}

A more promising avenue to avoid the counterexample of \Cref{rem:Bousfieldcounterexample} is to focus on spaces $X$ and $Y$ of which the $v_i$-periodic homotopy groups vanish for $i < n$. In that case we do get the following much sharper statement in the context of infinite loop spaces (but arbitrary maps between them):

\begin{thm}
\label{thm:maininfloop}
Let $E$ and $F$ be $p$-local spectra for which $\Loc_{n-1}^f E \simeq 0 \simeq \Loc_{n-1}^f F$. Let $f\colon \Loop^\infty E \to \Loop^\infty F$ be a map of pointed spaces. Then the following are equivalent:
\begin{itemize}
\item[(1)] The map $f$ is a $T(n)$-equivalence (resp. a parametric $T(n)$-equivalence).
\item[(2)] The truncation $\tau_{\leq n} f$ (resp. the truncation $\tau_{\leq n+1} f$) is an equivalence and $f$ is a $v_n$-periodic equivalence at any basepoint. 
\end{itemize}
\end{thm}


\begin{cor}
\label{cor:KuhnT(n)}
Suppose $E$ is a spectrum with $\Loc_{n-1}^f E \cong 0$. Then for any $m > n$ the map
\[
\tau_{>m}\Loop^\infty E \to \tau_{>n}\Loop^\infty E
\]
is a $T(n)$-equivalence.
\end{cor}
This last corollary is the telescopic analog of a result of Kuhn \cite[Theorem 2.21, Example 2.24]{kuhnAQG} for the Morava $K$-theories, proved through a different technique. Note that the counterexample of \cite[Lemma 8.8]{Bou01} shows that the corollary cannot generally hold if one removes the hypothesis $\Loc_{n-1}^f E \cong 0$. We will also deduce the following comparison between the $\Loc_n^f$-localizations of $E$ and of $\Loop^\infty E$, which the reader might compare to Bousfield's \cite[Theorem 8.1]{Bou01}. (Here we are using $\Loc_n^f$ to denote localization with respect to the homology theory $T(0) \oplus \cdots \oplus T(n)$, both for spaces and for spectra.)

 
\begin{cor}
\label{cor:T(n)localinfiniteloop}
Let $E$ be a $p$-local spectrum with $\Loc_{n-1}^f E \cong 0$. Then the square
\[
\begin{tikzcd}
\Loc_n^f\Loop^\infty E \ar{r}\ar{d} & \Loop^\infty \Loc_n^f E \ar{d} \\
\tau_{\leq n}\Loop^\infty E \ar{r} & \tau_{\leq n}\Loop^\infty L_n^f E
\end{tikzcd}
\]
is a pullback in the $\infty$-category of pointed spaces.
\end{cor}

Motivated by \Cref{thm:maininfloop} we pose the following:

\textbf{Question 1.} Suppose $f\colon X \to Y$ is a map of spaces and that the $v_i$-periodic homotopy groups of $X$ and $Y$ vanish for $i < n$ at any basepoint. Is the map $f$ a parametric $T(n)$-equivalence if and only if the truncation $\tau_{\leq n+1} f$ is an equivalence and $f$ is a $v_n$-periodic equivalence?

In \Cref{subsec:generalizations} we discuss how several aspects of the proof of \Cref{thm:maininfloop} may be generalized from infinite loop spaces to arbitrary spaces, using ideas from the theory of spectral Lie algebras. From there we give a positive answer to Question 1 in some special cases, but we do not settle it in general.

\subsubsection{The role of periodic homology equivalences of spaces}

The fact that \Cref{thm:mainhtpytohomology} does not admit a version at a single height $n$ (cf. \Cref{rmk:multipleheights}) can be seen as a consequence of the following result of Bousfield \cite[Theorem 1.1]{Bou99MK}:

\begin{thm}[Bousfield]
\label{thm:Bousfield}
Let $f\colon X \to Y$ be a $K(n)$-equivalence of spaces. Then $f$ is also a $K(i)$-equivalence for $0 < i < n$.
\end{thm}

This statement contrasts starkly with stable homotopy theory: for a map $f$ of \emph{spectra}, the statement above is well-known in case $X$ and $Y$ are assumed to be finite, but completely false in general. The fact that it holds for arbitrary spaces suggests that localizing at $K(n)$-homology (or at $T(n)$-homology) should \emph{not} be thought of as making the homotopy theory of spaces `monochromatic': indeed, it still records information about all heights $i < n$ (except 0) and is therefore perhaps best considered as a `transchromatic' localization.

\begin{rem}
\label{rmk:vnversusTn}
A statement like \Cref{thm:Bousfield} does not hold for periodic homotopy groups, again illustrating the difference between periodic homotopy and periodic homology. Indeed, if $X$ is a pointed space, then the nullification map $\eta\colon X \to \mathbf{P}_{V_n} X$ with respect to a type $n$ suspension $V_n$ (see \Cref{subsec:localizationnullification}) kills the $v_n$-periodic homotopy groups but is an isomorphism on $v_i$-periodic homotopy groups for $i < n$. Observe that the map $\eta$ is also a $T(i)$-homology isomorphism for $i < n$, simply because $V_n$ is $T(i)$-acyclic. In particular, if one starts with a highly-connected space $X$ with nonvanishing $K(i)$-homology, this yields an example of a highly connected space $\mathbf{P}_{V_n} X$ with trivial $v_n$-periodic homotopy groups, but nonvanishing $K(i)$-homology, hence nonvanishing $K(n)$-homology by \Cref{thm:Bousfield}, hence nonvanishing $T(n)$-homology. Such a space was used in \Cref{rmk:multipleheights} above to contradict the statement (*).
\end{rem}

Of course these considerations naturally lead to the following question:

\textbf{Question 2.} Does Bousfield's \Cref{thm:Bousfield} hold with $K$ replaced by $T$? More precisely, is a $T(n)$-equivalence of spaces also a $T(i)$-equivalence for $0 < i < n$?

Bousfield's proof of \Cref{thm:Bousfield} does not readily generalize to the telescopic case, so a new approach would be needed. If the answer to this question is affirmative, then we would have the following strengthenings of \Cref{thm:mainabs} and \Cref{thm:mainhomologytohtpy}:

\begin{thm}
\label{thm:mainabs2}
Suppose that the answer to Question 2 is yes. Let $X$ be a weakly $p$-local space with vanishing rational homology and let $n \geq 1$ be a natural number. Then $X \to *$ is a parametric $T(n)$-equivalence if and only if $X$ is $(n+1)$-connected and has vanishing $v_i$-periodic homotopy groups for every $i \leq n$.
\end{thm}

\begin{thm}
\label{thm:mainhomologytohtpy2}
Suppose that the answer to Question 2 is yes. If $f\colon X \to Y$ is a parametric $T(n)$-equivalence of weakly $p^\infty$-torsion spaces, then the truncation $\tau_{\leq n+1} f$ is an equivalence and $f$ is a $v_i$-periodic equivalence at every basepoint for any $i \leq n$.
\end{thm}

This stronger version of \Cref{thm:mainhomologytohtpy} would strengthen the parallel with \Cref{thm:mainhtpytohomology}. Indeed, the two results would be each other's converse except for the appearance of $\tau_{\leq n+1} f$ in one and an $(n+2)$-connected map $f$ in the other.

\subsubsection{Consequences for $T(n)$-equivalences}

We have discussed the limitations of the notion of $T(n)$-equivalence and replaced it by the more robust notion of parametric $T(n)$-equivalence. However, Bousfield observed that if one considers maps $f\colon X \to Y$ between $H$-spaces, the behavior of $T(n)$-equivalences is more controllable (see for instance \cite{Bou97}). Specializing our arguments to maps between $H$-spaces, we will deduce the following analog of \Cref{thm:mainhomologytohtpy}:

\begin{thm}
\label{thm:Hmainhomologytohtpy}
Suppose that $X$ and $Y$ are $H$-spaces whose homotopy groups are $p^\infty$-torsion in all degrees $\geq 1$. If $f\colon X \to Y$ is a $T(n)$-equivalence, then the truncation $\tau_{\leq n} f$ is an equivalence and $f$ is a $v_n$-periodic equivalence at every basepoint.
\end{thm}

Also, \Cref{thm:mainhtpytohomology} admits the following variation for ordinary $T(n)$-equivalences:

\begin{thm}
\label{thm:Hmainhtpytohomology}
If $f\colon X \to Y$ is an $(n+1)$-connected map that is a $v_i$-periodic equivalence at every basepoint, for all $i \leq n$, then $f$ is a $T(i)$-equivalence for all $i \leq n$.
\end{thm}

\subsubsection{Plan of this paper}

In \Cref{sec:homologyequiv} we develop the notion of parametric $E$-equivalence, for a spectrum $E$, and compare it to ordinary $E$-equivalences and Bousfield's notion of virtual $E$-equivalences. We also record several results on the $T(n)$-homology of Eilenberg--MacLane spaces and truncated spaces that follow from the work of Carmeli--Schlank--Yanovski \cite{CSY22}. \Cref{sec:periodichomotopy} is a survey of the notion of $v_n$-periodic homotopy groups of spaces, the Bousfield--Kuhn functor, and the theory of localization and nullification of spaces developed by Bousfield \cite{Bou94} and Farjoun \cite{Dro95}. We include this material because we need certain aspects of it that cannot be found in the literature in the precise form that we need them here. In \Cref{sec:homologyvshomotopy} we get to the proofs of our first set of main results, namely Theorems \ref{thm:mainabs}, \ref{thm:mainhomologytohtpy}, \ref{thm:mainhtpytohomology}, \ref{thm:Hmainhomologytohtpy}, and \ref{thm:Hmainhtpytohomology}. The final \Cref{sec:infiniteloop} is focused on the special case of infinite loop spaces and contains the proofs of \Cref{thm:maininfloop} and its corollaries. Some of the material on $T(n)$-local commutative ring spectra in that section may be of independent interest.

\subsubsection{Acknowledgments}
GH thanks Jeremy Hahn for pointing out the relevance of the ambidexterity idempotent for the arguments in \Cref{sec:infiniteloop} and Maite Carli for comments on a draft. GH acknowledges the support of the ERC through Starting Grant no. 950048 and the NWO through Vidi grant no. 233.093.


  \section{Several notions of homology equivalence}
\label{sec:homologyequiv}

The aim of this section is to develop the notion of \emph{parametric $E$-equivalence} of spaces, for a homology theory $E$, as well as the corresponding localization of the $\infty$-category $\mathcal{S}$ of spaces. We will investigate some basic properties, as well as the relation to ordinary $E$-equivalences, in \Cref{subsec:parametricequivalences}. This includes a proof of the existence of a parametric localization functor. In \Cref{subsec:parametricT(n)} we discuss features that are particular to the case where $E$ is a telescope $T(n)$. Finally, in \Cref{subsec:virtualequivalences} we establish the relation between our parametric $E$-equivalences and Bousfield's \emph{virtual $E$-equivalences}. 

\subsection{Parametric homology equivalences}
\label{subsec:parametricequivalences}

Let $E$ be a spectrum. As usual, we call a map $f\colon X \to Y$ of spaces (resp. of spectra) an \emph{$E$-equivalence} if $E \otimes \susp^\infty_+ f$ (resp. $E \otimes f$) is an equivalence of spectra. Bousfield has proved that the localization of spectra at the $E$-equivalences exists \cite{bousfieldspectra}. To be precise, a spectrum $Z$ is \emph{$E$-local} if for any $E$-equivalence of spectra $f\colon X \to Y$, the induced map $\mathrm{Map}(Y,Z) \to \mathrm{Map}(X,Z)$ is an equivalence of spaces. (Alternatively, one can require $\mathrm{Map}(A,Z) \cong 0$ for every $E$-acyclic spectrum $A$.) We will write $\Sp_E$ for the full subcategory of $\Sp$ on the $E$-local spectra. Then the inclusion $\Sp_E \to \Sp$ admits a left adjoint $\Loc_E\colon\Sp \to \Sp_E$ with the property that for any spectrum $X$, the unit map $X \to \Loc_E X$ is an $E$-equivalence. It follows from this that the functor $\Loc_E$ exhibits $\Sp_E$ as the localization of $\Sp$ at the $E$-equivalences.

\begin{defn}
A map $f\colon X \to Y$ of spaces is a \emph{parametric $E$-equivalence} if the induced functor
\[
f^*\colon \Fun(Y,\Sp_E) \rightarrow \Fun(X,\Sp_E)
\]
is an equivalence of $\infty$-categories.
\end{defn}

Let us immediately record some basic observations.

\begin{lem}
\label{lem:paramordinaryEequiv}
A parametric $E$-equivalence is in particular an $E$-equivalence.
\end{lem}
\begin{proof}
Let $f\colon X \to Y$ be a parametric $E$-equivalence and write $f_!\colon \Fun(X,\Sp_E) \to \Fun(Y,\Sp_E)$ for the left adjoint of $f^*$. By slight abuse of notation we write $\Loc_E\mathbb{S} \in \Fun(Y,\Sp_E)$ for the constant local system with value the $E$-localization of the sphere spectrum. Then the counit map $f_!f^*(\Loc_E\mathbb{S}) \to \Loc_E\mathbb{S}$ is an equivalence by hypothesis. The following commutative diagram then shows that $f$ is an $E$-equivalence:
\[
\begin{tikzcd}
\Loc_E\susp^\infty_+ X \ar{d}{\simeq}\ar{r}{\Loc_E\susp^\infty_+ f} & \Loc_E\susp^\infty_+ Y \ar{d}{\simeq} \\
\mathrm{colim}_{Y} f_!f^* \Loc_E\mathbb{S} \ar{r}{\simeq} & \mathrm{colim}_{Y} \Loc_E\mathbb{S}. 
\end{tikzcd}
\]
For the left vertical arrow we have used that for any local system $F \in \Fun(X,\Sp_E)$, there is an evident natural equivalence $\mathrm{colim}_X F \simeq \mathrm{colim}_Y f_! F$.
\end{proof}

\begin{lem}
\label{lem:parametrictau1equiv}
Let $f\colon X \to Y$ be a map that induces a bijection on components $\pi_0 f\colon \pi_0 X \xrightarrow{\cong} \pi_0 Y$ and a parametric $E$-equivalence from each component of $X$ to the corresponding component of $Y$. Then $f$ is a parametric $E$-equivalence. The converse is true whenever $E \not\simeq 0$. 
\end{lem}
\begin{proof}
The first part of the lemma is evident. For the converse, note first that the map $\pi_0 f$ is a retract of $f$ itself, hence a parametric $E$-equivalence whenever $f$ is. But that means
\[
(\pi_0 f)^*\colon \prod_{\pi_0 Y} \Sp_E \to \prod_{\pi_0 X} \Sp_E
\]
is an equivalence of $\infty$-categories. If $\Sp_E$ is nontrivial, this can only happen if $\pi_0 f$ is a bijection. Now if $X_0$ is a component of $X$ and $Y_0$ the corresponding component of $Y$, the restricted map $f_0\colon X_0 \to Y_0$ is a retract of $f$ and hence a parametric $E$-equivalence.
\end{proof}

The case of connected $X$ and $Y$ can be described in the following terms:

\begin{lem}
\label{lem:parametricvsordinary}
If $f\colon X \to Y$ is a pointed map of pointed spaces that is a parametric $E$-equivalence, then $\Loop f$ is an $E$-equivalence. The converse is true if $X$ and $Y$ are connected.
\end{lem}
\begin{proof}
By \Cref{lem:parametrictau1equiv} we may work one component at a time, so there is no loss of generality in assuming $X$ and $Y$ are connected throughout. Write $x\colon * \to X$ for the basepoint of $X$. Then the adjunction
\[
\begin{tikzcd}
\Sp_E \ar[shift left]{r}{x_!} & \Fun(X,\Sp_E) \ar[shift left]{l}{x^*}  
\end{tikzcd}
\]
is monadic; indeed, the right adjoint $x^*$ preserves colimits and is conservative since $X$ is connected. Since $x^*x_!(\Loc_E\mathbb{S}) \simeq \Loc_E\susp^\infty_+ \Loop X$, we conclude that $x^*$ induces an equivalence of $\infty$-categories
\[
\Fun(X,\Sp_E) \xto{\simeq} \mathrm{LMod}_{\Loc_E\susp^\infty_+ \Loop X}(\Sp_E).
\]
A similar statement applies to $Y$. Under these identifications, the pushforward $f_!$ can be identified with the functor
\[
\Loc_E\susp^\infty_+ \Loop Y \otimes_{\Loc_E\susp^\infty_+ \Loop X} -\colon \mathrm{LMod}_{\Loc_E\susp^\infty_+ \Loop X}(\Sp_E) \to \mathrm{LMod}_{\Loc_E\susp^\infty_+ \Loop Y}(\Sp_E).
\]
From this the second part of the lemma is clear. For the converse, suppose that $f$ is a parametric $E$-equivalence, so that the functor above is an equivalence of $\infty$-categories. Then the unit map yields an equivalence 
\[
\Loc_E\susp^\infty_+ \Loop X \xto{\simeq} \Loc_E\susp^\infty_+ \Loop Y \otimes_{\Loc_E\susp^\infty_+ \Loop X} \Loc_E\susp^\infty_+ \Loop X \simeq \Loc_E\susp^\infty_+ \Loop Y.
\]
\end{proof}

\begin{cor}
\label{cor:parametrictau1equiv}
Suppose $E \not\simeq 0$. Then if $f\colon X \to Y$ is a parametric $E$-equivalence, the truncation $\tau_{\leq 1}f$ is an equivalence of spaces.
\end{cor}
\begin{proof}
\Cref{lem:parametrictau1equiv} implies that $\pi_0 f$ is a bijection. For any choice of basepoint $x$ of $X$, \Cref{lem:parametricvsordinary} implies that the resulting map $\Loop f$ is an $E$-equivalence. If $E \not\simeq 0$ this implies that $\pi_0(\Loop f)$ is bijective. In other words, $f$ induces an isomorphism $\pi_1(X,x) \cong \pi_1(Y, f(x))$.
\end{proof}

We will conclude this section with a discussion of the existence of localizations at the parametric $E$-equivalences. First of all, recall that Bousfield \cite{Bou75} has shown that the localization of $\mathcal{S}$ at the $E$-equivalences exists. Indeed, a space $Z$ is said to be \emph{$E$-local} if for every $E$-equivalence of spaces $X \to Y$, the induced map $\mathrm{Map}(Y,Z) \to \mathrm{Map}(X,Z)$ is an equivalence. Then Bousfield shows that any space admits an \emph{$E$-localization} $X \to \Loc_E X$, which is an $E$-equivalence into an $E$-local space. Writing $\mathcal{S}_E$ for the full subcategory of $\mathcal{S}$ on the $E$-local spaces, there is an adjoint pair
\[
\begin{tikzcd}
\mathcal{S} \ar[shift left]{r}{\Loc_E} & \mathcal{S}_E \ar[shift left]{l} 
\end{tikzcd}
\]
where the right adjoint is the inclusion. The functor $\Loc_E$ exhibits the $\infty$-category $\mathcal{S}_E$ as the localization of $\mathcal{S}$ at the $E$-equivalences. It will be relevant to us that $\Loc_E$ preserves finite products; this is an immediate consequence of the fact that a finite product of $E$-equivalences is an $E$-equivalence.

In analogy with the above, we will call a space $Z$ \emph{parametrically $E$-local} if for every parametric $E$-equivalence of spaces $X \to Y$, the induced map $\mathrm{Map}(Y,Z) \to \mathrm{Map}(X,Z)$ is an equivalence. We denote by $\Loc_E^{\mathrm{par}}\mathcal{S}$ the full subcategory of $\mathcal{S}$ on the parametrically $E$-local spaces. We now establish the existence of localization with respect to parametric $E$-equivalences:

\begin{prop}
\label{prop:parametriclocalization}
\begin{itemize}
\item[(1)] A connected space $Z$ is parametrically $E$-local if and only if the loop space (at some basepoint) $\Loop Z$ is $E$-local.
\item[(2)] Any space admits a parametric $E$-equivalence into a parametrically $E$-local space.
\item[(3)] The inclusion of $\Loc_E^{\mathrm{par}}\mathcal{S}$ into $\mathcal{S}$ admits a left adjoint.
\end{itemize}
\end{prop}

We will denote the left adjoint functor of item (3) by $\Loc_E^{\mathrm{par}}$. It follows from the proposition that this functor exhibits $\Loc_E^{\mathrm{par}}\mathcal{S}$ as the localization of $\mathcal{S}$ at the parametric $E$-equivalences.

\begin{proof}[Proof of \Cref{prop:parametriclocalization}]
(1). It is straightforward to see that a connected space $Z$ is parametrically $E$-local if and only if for every parametric $E$-equivalence $f\colon X \to Y$ of \emph{connected} spaces, the induced map $f^*\colon \mathrm{Map}(Y,Z) \to \mathrm{Map}(X,Z)$ is an equivalence. Pick basepoints $x$ of $X$ and $z$ of $Z$ and consider the fibration
\[
\mathrm{Map}_*(X,Z) \to \mathrm{Map}(X,Z) \xrightarrow{\mathrm{ev}_x} Z.
\]
Doing the same for $Y$, we see that $\mathrm{Map}(Y,Z) \to \mathrm{Map}(X,Z)$ is an equivalence if and only if $\mathrm{Map}_*(Y,Z) \to \mathrm{Map}_*(X,Z)$ is an equivalence. There is a commutative square
\[
\begin{tikzcd}
 \mathrm{Map}_*(Y,Z) \ar{r}{f^*}\ar{d}{\simeq}[swap]{\Loop} & \mathrm{Map}_*(X,Z) \ar{d}{\simeq}[swap]{\Loop} \\
 \mathrm{Map}_{\mathbb{E}_1}(\Loop Y,\Loop Z) \ar{r}{(\Loop f)^*} & \mathrm{Map}_{\mathbb{E}_1}(\Loop X,\Loop Z)
\end{tikzcd}
\]
where the subscript $\mathbb{E}_1$ indicates the space of $\mathbb{E}_1$-maps (or equivalently loop maps). The map $\Loop f$ is an $E$-equivalence by \Cref{lem:parametricvsordinary}. Hence if $\Loop Z$ is $E$-local, the bottom map is an equivalence, hence so is the top and we conclude that $Z$ is parametrically $E$-local. For the converse, assume that $Z$ is parametrically $E$-local and that $g\colon U \to V$ is an $E$-equivalence of pointed spaces. To see that $\mathrm{Map}_*(V,\Loop Z) \to \mathrm{Map}_*(U,\Loop Z)$ is an equivalence, it suffices to argue that $\susp g$ is a parametric $E$-equivalence. By \Cref{lem:parametricvsordinary} this is the case precisely if $\Loop\susp g$ is an $E$-equivalence. The functor $\Loop\susp$ is easily seen to preserve $E$-equivalences by the James splitting of $\susp^\infty_+\Loop\susp f$.

(2). It is straightforward to see that a space is parametrically $E$-local if and only if each of its components is, so without loss of generality we can prove the statement for a connected space $Z$. Pick a basepoint of $Z$ and consider the map 
\[
Z \simeq B(\Loop Z) \to B(\Loc_E\Loop Z)
\]
arising from the $E$-localization of $\Loop Z$. Then the space $B(\Loc_E\Loop Z)$ is parametrically $E$-local by item (1) and the map above is a parametric $E$-equivalence by \Cref{lem:parametricvsordinary}.

(3). It suffices to check that for every $X \in \mathcal{S}$, the functor
\[
\mathrm{Map}(X,-)\colon \Loc_E^{\mathrm{par}}\mathcal{S} \to \mathcal{S}
\]
is corepresentable. For this one uses (2) for the existence of a parametric $E$-equivalence $X \to Y$, with $Y$ parametrically $E$-local, and observes that $Y$ corepresents the functor above.
\end{proof}

\subsection{Parametric $T(n)$-equivalences}
\label{subsec:parametricT(n)}

In the previous section we introduced the notion of parametric $E$-equivalence. In this paper we will mainly be interested in the case $E = T(n)$. Let us first briefly analyze the special case $n=0$, corresponding to rational homology, and then proceed to higher heights.

\begin{prop}
\label{prop:rationalparametricequiv}
A map $f\colon X \to Y$ is a parametric $\mathbb{Q}$-equivalence if and only if $\tau_{\leq 1} f$ is an equivalence and for any choice of basepoint $x \in X$ and $k \geq 2$, the map $f$ induces an isomorphism 
\[\pi_k(X,x) \otimes \mathbb{Q} \xto{\cong} \pi_k(Y,f(x)) \otimes \mathbb{Q}.\]
\end{prop}
\begin{rem}
Relatedly, it follows from \Cref{prop:parametriclocalization} that a space $X$ is parametrically $\mathbb{Q}$-local if and only if at every basepoint $x \in X$, the homotopy groups $\pi_k(X,x)$ are rational vector spaces for any $k \geq 2$.
\end{rem}
\begin{proof}
Suppose that $f$ is a parametric $\mathbb{Q}$-equivalence. Then $\tau_{\leq 1} f$ is an equivalence by \Cref{cor:parametrictau1equiv}. Moreover, for any choice of basepoint in $X$ the induced map $\Loop f\colon \Loop X \to \Loop Y$ is a rational homology equivalence by \Cref{lem:parametricvsordinary}, hence also a rational homotopy equivalence by the Hurewicz theorem. (Here we rely on the fact that any loop space is simple, i.e., has an abelian fundamental group that acts trivially on higher homotopy groups.) For the converse, suppose that $f$ induces an equivalence $\tau_{\leq 1} f$ and isomorphisms on rational homotopy groups in dimensions $\geq 2$. By \Cref{lem:parametrictau1equiv} we may work one connected component at a time and thus reduce to the case that $X$ and $Y$ are connected. Consider the diagram of fiber sequences
\[
\begin{tikzcd}
\tau_{>0} \Loop X \ar{d}\ar{r} & \Loop X \ar{d}\ar{r} & \pi_1 X \ar{d}{\cong} \\
\tau_{>0} \Loop Y \ar{r} & \Loop Y \ar{r} & \pi_1 Y. 
\end{tikzcd}
\]
The left-most map is a rational homotopy equivalence of simple spaces and hence a rational homology equivalence. It follows that the middle map is a rational homology equivalence as well, which suffices by \Cref{lem:parametricvsordinary}.
\end{proof}

For the remainder of this section we fix a choice of prime number $p$ and height $n \geq 1$. We will now record several results that are specific to the $T(n)$-local setting, in particular on the behavior of parametric $T(n)$-localization on truncated spaces. Most of these results follow directly from a combination of results of Bousfield \cite{Bou01} with the work of Carmeli--Schlank--Yanovski \cite{CSY22}. 

Recall that an abelian group $A$ is said to be \emph{derived $p$-complete} if it is $p$-complete when considered as an object of the derived $\infty$-category $\mathcal{D}(\mathbb{Z})$. Equivalently, $A$ is derived $p$-complete if and only if $\mathrm{Hom}(\mathbb{Q}_p/\mathbb{Z}_p, A) = 0$ and the canonical homomorphism 
\[
A \to \mathrm{Ext}(\mathbb{Q}_p/\mathbb{Z}_p, A)
\]
is an isomorphism. (The map is induced by the short exact sequence $\mathbb{Z} \to \mathbb{Z}{[1/p]} \to \mathbb{Q}_p/\mathbb{Z}_p$.) This is the condition that Bousfield--Kan refer to as being \emph{Ext-$p$-complete}. 

\begin{prop}
\label{prop:T(n)localEMspaces}
\begin{itemize}
\item[(1)] For $A$ an abelian group and $k \geq n+2$, the Eilenberg--MacLane space $K(A,k)$ is $T(n)$-acyclic. If $A$ is torsion, then $K(A,n+1)$ is $T(n)$-acyclic as well.
\item[(2)] A truncated $H$-space $X$ is $T(n)$-local if and only if $\pi_k(X,x)$ is derived $p$-complete for any $x \in X$ and $1 \leq k \leq n+1$, torsionfree for $k = n+1$, and vanishes for $k > n+1$.
\end{itemize}
\end{prop}


\begin{proof}
(1). It is proved in~\cite[Theorem~E]{CSY22} that the space $K(\mathbb{Z}/p,n+1)$ is $T(n)$-acyclic. Then it follows easily by induction that for any finite abelian group $B$, the space $K(B,n+1)$ is also $T(n)$-acyclic. Finally, a general torsion abelian group $A$ is a filtered colimit of finite abelian groups, so that $K(A,n+1)$ is $T(n)$-acyclic as well. Since $K(\mathbb{Z}_p,n+2)$ is the $p$-completion of $K(\mathbb{Q}_p/\mathbb{Z}_p,n+1)$, it is also $T(n)$-acyclic. The same inductive argument as before now shows that $K(A,n+2)$ is $T(n)$-acyclic for any abelian group. The same is then true of the higher Eilenberg--MacLane spaces $K(A,k)$ with $k \geq n+2$ since they can be built out of $K(A,n+2)$ via finite products and colimits.

(2). Bousfield~\cite[Theorem 7.2]{Bou82} proves that an $H$-space as in the statement is $K(n)$-local, thus in particular $T(n)$-local. Conversely, if $X$ is a truncated $H$-space that is $T(n)$-local, then in particular it is $p$-complete. Furthermore, if $\widehat{\tau}_{\leq n+1} X$ denotes the `modified Postnikov section' of $X$ (cf. \cite[Section 7.1]{Bou82}) which is obtained from the usual truncation $\tau_{\leq n+1} X$ by replacing $\pi_{n+1} X$ with its maximal torsionfree derived $p$-complete quotient, then the map $t\colon X \to \widehat{\tau}_{\leq n+1} X$ is a $T(n)$-equivalence by (1); indeed, its fiber is $T(n)$-acyclic. Furthermore, $\widehat{\tau}_{\leq n+1} X$ is $T(n)$-local by the first half of the argument and therefore $t$ is an equivalence.
\end{proof}

From \Cref{prop:T(n)localEMspaces} we can immediately deduce the following:

\begin{prop}
\label{prop:T(n)equivtruncation}
Let $f\colon X \to Y$ be a map between $H$-spaces and assume that the homotopy groups $\pi_k X$ and $\pi_k Y$ are $p^\infty$-torsion for $k \geq 1$. If $f$ is a $T(n)$-equivalence, then $\tau_{\leq n} f$ is an equivalence.
\end{prop}
\begin{proof}
Without loss of generality we may suppose that $X$ and $Y$ are connected. For any connected nilpotent space $Z$ we will write $\mathrm{tors}_p Z$ for its $p^\infty$-torsion part, defined as the fiber of the rationalization map $Z_p^{\wedge} \to L_\mathbb{Q} Z_p^{\wedge}$. We note that for a connected $H$-space $Z$, its $p^\infty$-torsion part $\mathrm{tors}_p Z$ is $n$-truncated if and only if $Z_p^{\wedge}$ is $(n+1)$-truncated and has torsionfree $\pi_{n+1}$. By \Cref{prop:T(n)localEMspaces}(2), this is equivalent to the statement that $Z_p^{\wedge}$ is $T(n)$-local. 

Since $X$ and $Y$ have $p^\infty$-torsion homotopy groups, it follows that for any $H$-space $Z$ we have an equivalence $\mathrm{Map}_*(X, \mathrm{tors}_p Z) \xrightarrow{\simeq} \mathrm{Map}_*(X, Z_p^{\wedge})
$ and similarly for $Y$. In particular, the comments in the preceding paragraph combined with the hypothesis that $f$ is a $T(n)$-equivalence show that if $Z$ is any $n$-truncated $H$-space with $p^\infty$-torsion homotopy groups, then 
\[
f^*\colon \mathrm{Map}_*(Y, Z) \to \mathrm{Map}_*(X, Z)
\]
is an equivalence. By the Yoneda lemma and the fact that $\mathrm{Map}_*(\tau_{\leq n} X, Z) \simeq \mathrm{Map}_*(X, Z)$, it follows that the map $\tau_{\leq n} f$ is an equivalence.
\end{proof}
\begin{rem}
We expect a version of \Cref{prop:T(n)equivtruncation} to hold true without the $H$-space hypothesis. Since we do not need it, we will not pursue this version here.
\end{rem}

\begin{cor}
\label{cor:T(n)equivtruncation}
Let $f\colon X \to Y$ be a parametric $T(n)$-equivalence between weakly $p^\infty$-torsion spaces. Then $\tau_{\leq n+1} f$ is an equivalence.
\end{cor}
\begin{proof}
As usual we may work one connected component at a time and assume that $X$ and $Y$ are connected. Picking a basepoint of $X$ we find that $\Loop f$ is a $T(n)$-equivalence between spaces satisfying the hypotheses of \Cref{prop:T(n)equivtruncation} and the conclusion follows.
\end{proof}

\subsection{Virtual $T(n)$-equivalences}
\label{subsec:virtualequivalences}

This section serves to clarify the relation between our notion of parametric $T(n)$-equivalence and Bousfield's notion of virtual $T(n)$-equivalence \cite[Section 11]{Bou97}, as well as to recall several useful properties of virtual $T(n)$-equivalences that will be needed later.

\begin{defn}\label{def:virtual-homology-equivalence}
A map $f\colon X \to Y$ of pointed spaces is a \emph{virtual $T(n)$-equivalence} if the map
\[
\tau_{> i}\Loc_{T(n)}(\Loop X) \to \tau_{> i}\Loc_{T(n)}(\Loop Y)
\]
is an equivalence for $i \gg 0$.
\end{defn}

In \cite[Theorem 11.10]{Bou97}, Bousfield proves that the notion of virtual $T(n)$-equivalence admits the following stronger characterization:

\begin{thm}[Bousfield]
\label{thm:virtualT(n)}
Let $f\colon X\to Y$ be a map of pointed spaces. Then the following are equivalent:
\begin{itemize}
\item[(1)] The map $f$ is a virtual $T(n)$-equivalence.
\item[(2)] For any $k \geq 1$ and $i \geq n+2$, the map $\tau_{> i}(\Loop^k X) \to \tau_{> i}(\Loop^k Y)$ is a $T(n)$-equivalence.
\item[(3)] There exist $k \geq 1$ and $i \geq n+2$ so that $\tau_{> i}(\Loop^k X) \to \tau_{> i}(\Loop^k Y)$ is a $T(n)$-equivalence.
\end{itemize}
\end{thm}

We need several further properties, also proved by Bousfield on the way to establishing \Cref{thm:virtualT(n)}.

\begin{prop}[{\cite[Theorem 11.3]{Bou97}}]
\label{prop:virtualequivHspace}
A $T(n)$-equivalence of $H$-spaces is a virtual $T(n)$-equivalence.
\end{prop}


The following is \cite[Theorem 11.4]{Bou97}, but is also a formal consequence of \Cref{thm:virtualT(n)}:

\begin{prop}
\label{prop:virtual2outof3}
Consider a map between fiber sequences of pointed spaces:
\[
\begin{tikzcd}
F \ar{r}\ar{d} & X \ar{r}\ar{d} & Y \ar{d} \\
F' \ar{r} & X' \ar{r} & Y'.
\end{tikzcd}
\]
If two out of the three vertical maps are virtual $T(n)$-equivalences, then so is the third.
\end{prop}


Finally we discuss the relation between parametric and virtual $T(n)$-equivalences. We have the following simple observation:

\begin{lem}
\label{lem:parametricimpliesvirtual}
A parametric $T(n)$-equivalence $f\colon X \to Y$ of pointed spaces is also a virtual $T(n)$-equivalence.
\end{lem}
\begin{proof}
According to \Cref{lem:parametricvsordinary} the map $\Loc_{T(n)}(\Loop f)$ is an equivalence. Hence it is in particular an equivalence in high degrees.
\end{proof}

This statement admits the following converse:

\begin{prop}
\label{prop:virtualvsparametric}
Let $f$ be a virtual $T(n)$-equivalence of pointed connected spaces.
\begin{itemize}
\item[(1)] If the fiber $\mathrm{fib}(f)$ is $n$-connected and $\pi_{n+1}(\mathrm{fib}(f))$ is torsion, then $f$ is a $T(n)$-equivalence.
\item[(2)] If the fiber $\mathrm{fib}(f)$ is $(n+1)$-connected and $\pi_{n+2}(\mathrm{fib}(f))$ is torsion, then $f$ is a parametric $T(n)$-equivalence. 
\end{itemize}
\end{prop}
\begin{proof}
The fiber $\mathrm{fib}(f)$ is virtually $T(n)$-acyclic by \Cref{prop:virtual2outof3}. By \Cref{thm:virtualT(n)} the space $\tau_{>n+2} \mathrm{fib}(f)$ is then $T(n)$-acyclic, so that the map $\mathrm{fib}(f) \to \tau_{\leq n+2} \mathrm{fib}(f)$ is a $T(n)$-equivalence. The space $\tau_{\leq n+2} \mathrm{fib}(f)$ can only have nontrivial homotopy groups in degrees $n+1$ and $n+2$, of which the first is torsion by hypothesis. It is therefore $T(n)$-acyclic as a consequence of \Cref{prop:T(n)localEMspaces}(1). We conclude that $\mathrm{fib}(f)$ is $T(n)$-acylic and $f$ is a $T(n)$-equivalence. For (2), we note that it suffices to prove that $\Loop f$ is a $T(n)$-equivalence. We start with the observation that also $\mathrm{fib}(\Loop f)$ is virtually $T(n)$-acyclic by \Cref{thm:virtualT(n)}. Then $\mathrm{fib}(\Loop f) \to \tau_{\leq n+2}(\mathrm{fib}(\Loop f)) = \Loop\tau_{\leq n+3}(\mathrm{fib}(f))$ is a $T(n)$-equivalence and the conclusion follows analogously.
\end{proof}




  \section{Periodic homotopy groups and localization of spaces}
\label{sec:periodichomotopy}

We will begin this section by recalling the definition of $v_n$-periodic homotopy groups and the construction of the telescopic functors of Bousfield and Kuhn. Then we review the relation between periodic homotopy groups and the theory of \emph{nullification} with respect to a given space $V$ developed by Bousfield \cite{Bou94} and Farjoun \cite{Dro95}. In particular, we discuss the (unstable) Bousfield classes of spaces. Apart from our synthesis we claim little originality here; all of the relevant ingredients were essentially already established by Bousfield \cite{Bou94, Bou01}, but we need to put them in a form that facilitates our work in later sections.

\subsection{Periodic homotopy groups}\label{sec:vn-periodic-groups}

As always we assume fixed a prime number $p$. Let $V_n$ be a finite pointed space of type $n$, equipped with a $v_n$ self-map $v\colon \susp^d V_n \to V_n$. We will take $n >0$, so that also $d > 0$. Then for any pointed space $X$ we may consider the \emph{$v$-periodic homotopy groups} of $X$ defined by $v^{-1}\pi_*\Map_*(V_n,X)$. In fact these homotopy groups arise as the homotopy groups of a spectrum \cite{Kuh08}. This spectrum is formed from the spaces $\Map_*(V_n,X)$ by using the `transition maps'
\[
\Map_*(V_n,X) \xto{v^*} \Map_*(\susp^d V_n, X) \simeq \Loop^d \Map_*(V_n,X).
\]
More precisely, Bousfield and Kuhn's telescopic functors arise as follows:

\begin{defn}
The \emph{telescopic functor} $\Phi_{V_n}\colon \mathcal{S}_* \to \Sp$ associated to $V$ is defined by
\[
\Phi_{V_n}(X) := \colim\left(\susp^\infty\Map_*(V_n, X)  \xto{v^*} \susp^{\infty-d} \Map_*(V_n, X) \xto{v^*} \susp^{\infty-2d} \Map_*(V_n, X) \xto{v^*} \cdots \right).
\]
\end{defn}

\begin{rem}
One easily verifies that the groups $\pi_*\Phi_{V_n}(X)$ are indeed isomorphic to the $v$-periodic homotopy groups of $X$ as defined above. Moreover, as the notation suggests, the functor $\Phi_{V_n}$ depends only on $V_n$ and not on the choice of self-map $v$. This is a consequence of the asymptotic uniqueness of $v_n$ self-maps \cite[Corollary 3.7]{HS98}.
\end{rem}

\begin{defn}
A map $f$ of pointed spaces is a \emph{$v_n$-periodic equivalence} if $\Phi_{V_n}(f)$ is an equivalence.
\end{defn}

The above is equivalent to demanding that $f$ induce an isomorphism on $v$-periodic homotopy groups. It is a consequence of the thick subcategory theorem \cite{HS98} that the notion of $v_n$-periodic equivalence only depends on $n$ and not on the choice of $V_n$ or of the self-map $v$.

\begin{rem}
\label{rem:BK}
The \emph{Bousfield--Kuhn functor} $\Phi_n\colon \mathcal{S}_* \to \Sp$ removes the dependence on $V$ (see \cite{Kuh08} for a detailed discussion). It is a functor satisfying the following two properties:
\begin{itemize}
\item[(1)] There is a natural equivalence $V_n^{\vee} \otimes \Phi_n \simeq \Phi_{V_n}$, where $V_n^{\vee}$ denotes the Spanier--Whitehead dual of $V$.
\item[(2)] The composite $\Sp \xto{\Loop^\infty} \mathcal{S}_* \xto{\Phi_n} \Sp$ is equivalent to the $T(n)$-localization functor.
\end{itemize}
In fact $\Phi_n$ is completely characterized by property (1), if that equivalence is also required to be natural in $V_n$. The homotopy groups of $\Phi_n$ can be interpreted as the `absolute' $v_n$-periodic homotopy groups, removing the dependence on a choice of `coefficients' $V_n$. In this paper we will not need $\Phi_n$ itself and only use the telescopic functors $\Phi_{V_n}$.
\end{rem}

Let us recall some basic properties of $\Phi_{V_n}$ for later use.

\begin{lem}
    \label{lem:basicpropertiesPhi}
    \item[(1)] For any pointed space $X$, the spectrum $\Phi_{V_n}(X)$ is $T(n)$-local.
    \item[(2)] The functor $\Phi_{V_n}$ preserves fiber sequences. In particular, a map of pointed spaces is a $v_n$-periodic equivalence if and only if $\Phi_{V_n}$ annihilates its fiber.
    \item[(3)] For a pointed space $X$, the $i$-connected cover $\tau_{>i} X \to X$ is a $v_n$-periodic equivalence for any $i \geq 0$.
    \item[(4)] The canonical map $\Map_*(V_n,X) \to \Loop^\infty \Phi_{V_n}(X)$ is a $v_n$-periodic equivalence.
\end{lem}
\begin{proof}
Item (1) is \cite[Theorem 1.1]{Kuh08}, item (2) is straightforward from the construction of the telescopic functor, and (3) follows from the combination of (2) and the observation that the $v$-periodic homotopy groups of any truncated space must vanish. To establish (4) we should argue that $\Phi_{V_n} \Map_*(V_n,X) \to \Phi_{V_n} \Loop^\infty \Phi_{V_n}(X)$ is an equivalence. Since $\Phi_{V_n}$ commutes with finite limits, it also commutes with the cotensoring of $\mathcal{S}_*$ and $\Sp$ by finite spaces, and therefore $\Phi_{V_n} \Map_*(V_n,X) \simeq V_n^{\vee} \otimes \Phi_{V_n}(X)$. By \Cref{rem:BK} we can also identify
\[
\Phi_{V_n} \Loop^\infty \Phi_{V_n}(X) \simeq V_n^{\vee} \otimes \Phi_n \Loop^\infty \Phi_{V_n}(X) \simeq V_n^{\vee} \otimes  \Phi_{V_n}(X)
\]
and under these identifications it is straightforward to verify that the map of (4) becomes the identity.
\end{proof}

The following will be a convenient criterion:

\begin{lem}\label{lem:periodic-equivalence-criterion}
Let $f\colon X \to Y$ be a map of pointed spaces and suppose that 
\[
\Map_*(V_n, f)\colon \Map_*(V_n, X) \to \Map_*(V_n,Y)
\]
is a $T(n)$-equivalence. Then $f$ is a $v_n$-periodic equivalence.
\end{lem}
\begin{proof}
Since $\Phi_{V_n}$ is constructed as a colimit of desuspensions of the functor $\susp^{\infty}\Map_*(V_n,-)$, it follows that $\Phi_{V_n}(f)$ is a $T(n)$-equivalence of spectra. But $\Phi_{V_n}(X)$ and $\Phi_{V_n}(Y)$ are $T(n)$-local by \Cref{lem:basicpropertiesPhi}(1), so that $\Phi_{V_n}(f)$ is an equivalence.
\end{proof}

\subsection{Localization and nullification of spaces}
\label{subsec:localizationnullification}

In this section we briefly review Bousfield and Farjoun's work \cite{Bou94,Dro95} on the localization $\Loc_f$ of the $\infty$-category $\mathcal{S}$ of spaces with respect to a fixed map $f\colon V \to W$. In the special case of a map $f\colon V \to *$, this localization is usually referred to as the \emph{nullification} with respect to the space $V$ and denoted $\mathbf{P}_V$. The main technical results we will need are Bousfield's \Cref{bousfield-classification} and its implications for the determination of the Bousfield classes of spaces.

\begin{defn}
Let $f\colon V \to W$ be a map of spaces.
\begin{enumerate}
\item A space $X$ is \emph{$\Loc_f$-local} if $f^*\colon \Map(W,X) \to \Map(V,X)$ is an equivalence.
\item A map $g\colon A \to B$ is an \emph{$\Loc_f$-equivalence} if for any $\Loc_f$-local space $X$, the map $g^*\colon \Map(B,X) \to \Map(A,X)$ is an equivalence. 
\item A map $i\colon C \to D$ is \emph{$f$-cellular} if it lies in the weakly saturated class generated by $f$, i.e., if it can be obtained from $f$ via pushouts, transfinite compositions, and retracts.
\end{enumerate}
\end{defn}

In the absolute case of a map $f\colon V \to *$, the terminology we  use is as follows: we say \emph{$\mathbf{P}_V$-local} instead of $\Loc_f$-local, \emph{$\mathbf{P}_V$-equivalence} instead of \emph{$\Loc_f$-equivalence} and \emph{$V$-cellular} instead of $f$-cellular.

\begin{rem}
It is common to say \emph{$f$-local} for what we call $\Loc_f$-local and \emph{$V$-null} or \emph{$V$-periodic} for what we call $\mathbf{P}_V$-local. We stick with the terminology introduced above to limit the amount of necessary terminology and notation and to avoid confusion with the kernel of the functor $\mathbf{P}_V$.
\end{rem}

We collect some of the basic existence results of the theory of localizations in the following (see e.g. \cite{Dro95}).

\begin{prop}
\label{prop:locbasics}
   \begin{enumerate}
    \item Any space $X$ admits an $f$-cellular map $X \to \Loc_f X$ so that $\Loc_f X$ is an $\Loc_f$-local space.
    \item Any $f$-cellular map is an $\Loc_f$-equivalence.
    \item The inclusion of the full subcategory $\Loc_f \mathcal{S}$ of $\Loc_f$-local pointed spaces into all of $\mathcal{S}$ admits a left adjoint, which we will denote $\Loc_f$.
   \end{enumerate}
\end{prop}
\begin{proof}
Item (1) is proved using the small object argument. Item (2) follows immediately from the observation that the $\Loc_f$-equivalences form a weakly saturated class. Item (3) follows from the observation that for any $X \in \mathcal{S}$, the functor
\[
\Map(X,-)\colon \Loc_f \mathcal{S} \to \mathcal{S}
\]
is corepresentable, namely by a space $\Loc_f X$ built as in item (1).
\end{proof}

\begin{example}
\label{ex:LfvsLE}
    Localization with respect to a homology theory $E$ is an example of $\Loc_f$-localization with respect to a suitably chosen map $f$ \cite[Section 2.5]{Bou97}. Indeed, $f$ can be taken to be the disjoint union of `all' $E$-equivalences; to be more precise, one considers $E$-equivalences between $\kappa$-small spaces, with $\kappa$ the cardinality of $\pi_* E$.
\end{example}

Let us observe that parametric $E$-localization is essentially also an example of $\Loc_f$-localization:

\begin{prop}
Let $E$ be a spectrum and let $f\colon A \to B$ be a map of pointed spaces so that $\Loc_f = \Loc_E$ (see \Cref{ex:LfvsLE}). Then a connected pointed space $X$ is $\Loc_{\susp f}$-local if and only if it is parametrically $E$-local.
\end{prop}
\begin{proof}
It is immediate from the definitions that $X$ is $\Loc_{\susp f}$-local if and only if $\Loop X$ is $\Loc_f$-local. The result now follows from \Cref{prop:parametriclocalization}(1).
\end{proof}

We will need several results on the interplay between suspensions and localizations, of which the first is the following (see \cite{farjounsmith}, \cite[Section 3]{Bou94}):

\begin{thm}[Bousfield, Farjoun]
\label{thm:LocLoop}
Let $f\colon V \to W$ be a map of pointed spaces. Then there is a natural equivalence of functors $\Loop \Loc_{\susp f} \simeq \Loc_f \Loop$ on the $\infty$-category $\mathcal{S}_*$ of pointed spaces.
\end{thm}

It will be useful to record the following observation on the interplay between $f$-localization and $E$-equivalences:

\begin{prop}
\label{prop:fequivEequiv}
    Let $E$ be a spectrum and $f\colon V \to W$ an $E$-equivalence of pointed spaces. For a map $g \colon X \to Y$ of pointed spaces we have:
    \begin{enumerate}
        \item 
        The map $g$ is an $E$-equivalence if and only if $\Loc_f(g)$ is an $E$-equivalence.
        \item 
        If $X$ and $Y$ are connected, then $g$ is a parametric $E$-equivalence if and only if $\Loc_{\susp f}(g)$ is a parametric $E$-equivalence.
    \end{enumerate} 
\end{prop}
\begin{proof}
For (1) it suffices to show that the maps $X \to \Loc_f X$ and $Y \to \Loc_f Y$ are $E$-equivalences. These maps are $f$-cellular by \Cref{prop:locbasics} and $f$-cellular maps are $E$-equivalences by the hypothesis that $f$ is an $E$-equivalence. For (2), we rely on \Cref{lem:parametricvsordinary} and the fact that $\Loop \Loc_{\susp f} X \simeq \Loc_f \Loop X$ (\Cref{thm:LocLoop}) to reduce to showing that $\Loop X \to \Loc_f \Loop X$ is an $E$-equivalence (and similarly for $Y$). This follows as in the first part of the proof.
\end{proof}


For the remainder of this section we focus on the case of localization with respect to a map $f\colon V \to *$, in which case the corresponding localization functor is denoted $\mathbf{P}_V$.

\begin{defn}
The \emph{Bousfield class} of a pointed space $V$, denoted by $\langle V \rangle$, is the class of all pointed spaces $X$ for which $\mathbf{P}_V(X) \simeq 0$. These Bousfield classes are partially ordered by inclusion: we write $\langle V \rangle \geq \langle W \rangle$ if $\mathbf{P}_W(X) \simeq 0$ implies $\mathbf{P}_V(X) \simeq 0$.
\end{defn}

\begin{rem}
\label{rem:Bousfieldclasses}
There are many alternative ways to characterize the relation $\langle V \rangle \geq \langle W \rangle$, of which we highlight the following:
\begin{itemize}
\item[(1)] Every $\mathbf{P}_V$-local space is also $\mathbf{P}_W$-local.
\item[(2)] $\mathbf{P}_V(W) \simeq 0$.
\end{itemize}
\end{rem}

\Cref{thm:LocLoop} implies the following elementary relations:

\begin{cor}
\label{cor:LocLoop}
Let $V$ and $W$ be pointed spaces. Then we have the following:
\begin{itemize}
    \item[(1)] If $W$ is connected, then $\langle \susp V \rangle \geq \langle W \rangle$ if and only if $\langle V \rangle \geq \langle \Loop W \rangle$.
    \item[(2)] If $V$ is connected, then $\langle \susp \Loop V \rangle \geq \langle V \rangle$.
    \item[(3)] $\langle V \rangle \geq \langle \Loop \susp V \rangle$.
\end{itemize}
\end{cor}
\begin{proof}
    For (1), first note that $\mathbf{P}_{\susp V} W$ is still connected. Now \Cref{thm:LocLoop} implies that $\mathbf{P}_{\susp V} W \simeq 0$ if and only if $\mathbf{P}_V(\Loop W) \simeq 0$. For (2), we apply item (1) to see that the inequality is equivalent to $\langle \Loop V \rangle \geq \langle \Loop V \rangle$. For (3), note that $\susp V$ is connected and apply (1) to see that the inequality is equivalent to $\langle \susp V \rangle \geq \langle \susp V \rangle$.
\end{proof}
\begin{rem}
\label{rmk:LocLoop}
Note that \Cref{cor:LocLoop} in particular implies $\langle \susp K(\mathbb{Z}/p,n) \rangle \geq \langle K(\mathbb{Z}/p,n+1) \rangle$.
\end{rem}

The following was proved by Bousfield \cite[Theorems 9.10, 9.11]{Bou94} and is a crucial tool in analyzing the Bousfield classes of spaces:

\begin{thm}[Bousfield]\label{bousfield-classification}
Let $n > 0$ be a natural number and let $X$ be a pointed space for which $\widetilde{H}_*(X)$ consists of $p^\infty$-torsion, $\widetilde{H}_i(X) = 0$ for $i < n$ and $\widetilde{H}_n(X) \neq 0$. Then for each $k \geq 1$ there is an equality of Bousfield classes
\[
\langle \susp X \rangle = \langle \susp^k X \rangle \vee \langle K(A, n+1) \rangle,
\]
where $A = \mathbb{Z}/p$ if $H_n(X;\mathbb{Z}/p) \neq 0$ and $A = \mathbb{Q}/\mathbb{Z}_{(p)}$ if $H_n(X;\mathbb{Z}/p) = 0$.
\end{thm}

We explicitly record the following simple consequence for later use:
\begin{cor}
\label{cor:bousfield-classification}
Let $X$ be as in \Cref{bousfield-classification}. Then for each $k \geq 1$ we have
 \[
\langle \susp X \rangle \leq \langle \susp^k X \rangle \vee \langle K(\mathbb{Z}/p, n+1) \rangle.
\]   
\end{cor}
\begin{proof}
Considering the two possible cases in \Cref{bousfield-classification}, it is clear that it suffices to prove $\langle K(\mathbb{Q}/\mathbb{Z}_{(p)}, n+1) \rangle \leq \langle K(\mathbb{Z}/p, n+1) \rangle$. The class of abelian groups $A$ for which $\mathbf{P}_{K(\mathbb{Z}/p, n+1)}(K(A,n+1)) \simeq 0$ is closed under extensions and directed colimits (cf. \cite[Lemma 7.4]{Bou94} and its proof). This class contains $\mathbb{Z}/p$ by definition, hence it also contains $\mathbb{Z}/p^k$ and $\varinjlim_k \mathbb{Z}/p^k \cong \mathbb{Q}/\mathbb{Z}_{(p)}$.
\end{proof}

Combining \Cref{bousfield-classification} with the Hopkins--Smith thick subcategory theorem, Bousfield easily deduces the following classification of finite suspensions \cite[Theorem 9.15]{Bou94}:

\begin{cor}
\label{cor:unstableBousfieldclasses}
Let $V$ and $W$ be finite $p$-local spaces of type $> 0$. Then the following statements are equivalent:
\begin{itemize}
\item[(1)] $\langle \susp V \rangle \geq \langle \susp W \rangle$,
\item[(2)] $\mathrm{type}(V) \leq \mathrm{type}(W)$ and $\mathrm{conn}(\susp V) \leq \mathrm{conn}(\susp W)$.
\end{itemize}
\end{cor}

Closely related to \Cref{thm:LocLoop} and \Cref{bousfield-classification}, Bousfield proves that the localization functors $\mathbf{P}_{\susp V}$ preserve finite limits up to `low-dimensional error terms'. We finish this section by giving a precise statement to be used later. First of all, observe that for pointed spaces $A$ and $X$, the mapping space $\Map_*(A,\mathbf{P}_{\susp V} X)$ is again $\mathbf{P}_{\susp V}$-local. Hence we find a natural map
\[
\eta\colon \mathbf{P}_{\susp V}\Map_*(A, X) \to \Map_*(A,\mathbf{P}_{\susp V} X).
\]

\begin{thm}[Bousfield, {\cite[Theorem 8.3]{Bou94}}]
\label{thm:Bousfieldfib}
Suppose $A$ is a finite pointed space and suppose that $V$ is a space for which $\widetilde{H}_*(V)$ consists of $p^\infty$-torsion, $\widetilde{H}_i(V) = 0$ for $i < n$, and $H_n(V;\mathbb{Z}/p) \neq 0$. Then the fiber $F$ of $\eta$ over any point of $\Map_*(A,\mathbf{P}_{\susp V} X)$ is $n$-truncated.
\end{thm}

\subsection{The $v_n$-periodization of spaces}

In this section we review the relation between $v_n$-periodic homotopy groups and the nullification $\mathbf{P}_{V_{n+1}}$ with respect to a finite space $V_{n+1}$ of type $n+1$. For a chosen such space $V_{n+1}$, we will write $W_{n+1} := V_{n+1} \vee K(\mathbb{Z}/p,n+1)$. (It is not difficult to deduce from \Cref{bousfield-classification} and \Cref{cor:unstableBousfieldclasses} that the Bousfield class of $W_{n+1}$ does not depend on the choice of $V_{n+1}$, although we will not need this.) Bousfield proved, roughly speaking, that a map of highly connected pointed spaces is a $\mathbf{P}_{\susp W_{n+1}}$-equivalence if and only if it is a $v_i$-periodic equivalence for every $i \leq n$ (see \cite[Theorem 11.10]{Bou94}). The precise version that we need is captured by the following two results:

\begin{thm}
\label{thm:Pvnperiodichtpy}
Let $V_{n+1}$ be a finite pointed space of type $n+1$ and $W_{n+1} = V_{n+1} \vee K(\mathbb{Z}/p,n+1)$. Then for any pointed space $X$ we have:
\begin{itemize}
\item[(1)] The map $X \to \mathbf{P}_{\susp W_{n+1}} X$ is a $v_i$-periodic equivalence for every $i \leq n$.
\item[(2)] If $X$ is an $H$-space, then $X \to \mathbf{P}_{W_{n+1}} X$ is a $v_i$-periodic equivalence for every $i \leq n$.
\end{itemize}
\end{thm}

In the opposite direction, we have the following:

\begin{thm}
\label{thm:periodichtpyPvn}
Let $f\colon X \to Y$ be a map of $p$-local pointed spaces that is a $v_i$-periodic equivalence for every $i \leq n$ and let $W_{n+1}$ be as in \Cref{thm:Pvnperiodichtpy}. If $f$ is $(n+2)$-connected, then $f$ is a $\mathbf{P}_{\susp W_{n+1}}$-equivalence.
\end{thm}

To prepare for the proofs, we first recall the relation between nullification with respect to $V_{n+1}$ and $W_{n+1}$. 

\begin{lem}
\label{lem:VnWn}
There exists $j \geq 1$ so that for any pointed space $X$, the natural map
\[
\mathbf{P}_{\susp V_{n+1}} X \to \mathbf{P}_{\susp W_{n+1}} X
\]
has $(n+j)$-truncated fiber. In particular, the map induces an equivalence on $(n+j+1)$-connected covers.
\end{lem}
\begin{proof}
It follows from \cite[Theorem 9.12]{Bou94} that there is some $j \geq 1$ so that $\langle \susp^j W_{n+1} \rangle = \langle \susp V_{n+1} \rangle$. (Alternatively, this can be deduced directly from \Cref{bousfield-classification}.) The fiber of the natural map
\[
\mathbf{P}_{\susp^j W_{n+1}} X \to \mathbf{P}_{\susp W_{n+1}} X
\]
is $(n+j)$-truncated by \cite[Theorem 7.2]{Bou94} or \cite{farjounsmith}.   
\end{proof}

\begin{proof}[Proof of \Cref{thm:Pvnperiodichtpy}]
We begin with item (1). Fix a finite pointed space $V_i$ of type $i$ with a $v_i$ self-map $v\colon \susp^d V_i \to V_i$. We aim to show that $\Phi_{V_i}(X) \to \Phi_{V_i}(\mathbf{P}_{\susp W_{n+1}} X)$ is an equivalence of spectra. The functor $\Phi_{V_i}$ is insensitive to passing to highly connected covers, so invoking \Cref{lem:VnWn} we may replace $\mathbf{P}_{\susp W_{n+1}} X$ with $\mathbf{P}_{\susp V_{n+1}} X$. Since $\Phi_{V_i}(X)$ and $\Phi_{V_i}(\mathbf{P}_{\susp V_{n+1}} X)$ are periodic, it will suffice to check that the associated map of infinite loop spaces is an equivalence, which by construction is the top horizontal map in the  following commutative diagram:
\[
\begin{tikzcd}
\mathrm{colim}_k \Map_*(\susp^{kd} V_i, X) \ar{r}\ar{d} & \mathrm{colim}_k \Map_*(\susp^{kd} V_i, \mathbf{P}_{\susp V_{n+1}} X) \\
\mathbf{P}_{\susp V_{n+1}}\mathrm{colim}_k \Map_*(\susp^{kd} V_i, X) \ar{r} & \mathrm{colim}_k \mathbf{P}_{\susp V_{n+1}} \Map_*(\susp^{kd} V_i, X). \ar{u}
\end{tikzcd}
\]
The bottom horizontal equivalence arises from the fact that $\mathbf{P}_{\susp V_{n+1}}$ preserves filtered colimits, since $V_{n+1}$ is finite. The space on the upper left, which is $\Loop^\infty \Phi_{V_i} X$, is $T(i)$-local by \Cref{lem:basicpropertiesPhi}. It is therefore also $\mathbf{P}_{\susp V_{n+1}}$-local, since $V_{n+1}$ is $T(i)$-acyclic, so that the left vertical map is an equivalence. Finally, the vertical map on the right is an equivalence in high degrees by \Cref{thm:Bousfieldfib}. It follows that the top horizontal map is an equivalence in high degrees, but then it is in fact an equivalence since the spaces involved have periodic homotopy groups.

Now suppose that $X$ is an $H$-space. Then it is a retract of $\Loop \susp X$, so for item (2) it will suffice to show that $\Loop \susp X \to \mathbf{P}_{W_{n+1}}(\Loop \susp X)$ is a $v_i$-periodic equivalence for $i \leq n$. According to \Cref{thm:LocLoop} we may identify this map with the image under $\Loop$ of the map $\susp X \to \mathbf{P}_{\susp W_{n+1}} \susp X$, which is a $v_i$-periodic equivalence by the first part of the proof.
\end{proof}

\begin{proof}[Proof of \Cref{thm:periodichtpyPvn}]
Write $F$ for the fiber of $f$ at an arbitrary basepoint $y$ of $Y$. By hypothesis this is an $(n+1)$-connected space with vanishing $v_i$-periodic homotopy groups for $i \leq n$. It will suffice to show that $\mathbf{P}_{\susp W_{n+1}} F$ is contractible. Observe that the space $\mathbf{P}_{\susp W_{n+1}} F$ is still $(n+1)$-connected (because $\susp W_{n+1}$ is) and has vanishing $v_i$-periodic homotopy groups by \Cref{thm:Pvnperiodichtpy}(1). Fix a sequence of finite pointed spaces $U_i$, with $0 \leq i \leq n+1$, such that:
\begin{itemize}
\item[(1)] $U_0$ is a sphere $S^k$,
\item[(2)] each $U_i$ admits a $v_i$ self-map $w_i\colon \susp^{d_i} U_i \to U_i$,
\item[(3)] $U_i = \mathrm{cof}(w_{i-1})$ for $1 \leq i \leq n+1$.
\end{itemize}
Such a sequence exists as a consequence of the Hopkins--Smith periodicity theorem. Now set $V_i := \susp^{n-i+1} U_i$. We will prove by a finite downward induction on $i$ that the space $\mathbf{P}_{\susp W_{n+1}} F$ is $\mathbf{P}_{\susp V_i}$-local. The base case $i = n+1$ follows from the fact that $\langle \susp W_{n+1} \rangle \geq \langle \susp V_{n+1} \rangle$ (where we may use \Cref{cor:unstableBousfieldclasses} and \Cref{bousfield-classification} to see that it does not matter which type $n+1$ space $V_{n+1}$ we picked previously). For the inductive step, consider the cofiber sequence
\[
U_{i+1} \to \susp^{d_i+1} U_i \xrightarrow{\susp w_i} \susp U_i
\]
and the resulting fiber sequence
\[
\mathrm{Map}_*(\susp V_i, \mathbf{P}_{\susp W_{n+1}} F) \xrightarrow{w_i^*} \mathrm{Map}_*(\susp^{d_i+1} V_i, \mathbf{P}_{\susp W_{n+1}} F) \to \mathrm{Map}_*(\susp V_{i+1}, \mathbf{P}_{\susp W_{n+1}} F).
\]
The right-hand term is contractible by the inductive hypothesis, so the first map is an equivalence. It follows that the natural map
\[
\mathrm{Map}_*(\susp V_i, \mathbf{P}_{\susp W_{n+1}} F) \to \mathrm{colim}_k \mathrm{Map}_*(\susp^{1 + kd_i} V_i, \mathbf{P}_{\susp W_{n+1}} F) = \Loop^\infty \Phi_{\susp V_i}(\mathbf{P}_{\susp W_{n+1}} F)
\]
is an equivalence. The right-hand side is contractible because $\mathbf{P}_{\susp W_{n+1}} F$ has vanishing $v_i$-periodic homotopy groups. Hence the left-hand side is contractible and $\mathbf{P}_{\susp W_{n+1}} F$ is $\mathbf{P}_{\susp V_i}$-local, as desired. Setting $i=0$ we conclude that $\mathbf{P}_{\susp W_{n+1}} F$ is $\mathbf{P}_{S^{k+n+1}}$-local, meaning it is $k+n$-truncated. At the same time, this space is $p^\infty$-torsion (since its $v_0$-periodic homotopy groups vanish), and $n+1$-connected. Also, $\mathbf{P}_{\susp W_{n+1}} F$ does not admit nontrivial maps from $\susp K(\mathbb{Z}/p,n+1)$ because $\langle \susp W_{n+1} \rangle \geq \langle \susp K(\mathbb{Z}/p,n+1) \rangle$. An easy downward induction on the Postnikov tower of $\mathbf{P}_{\susp W_{n+1}} F$ now shows that it is contractible, as desired.
\end{proof}


  \section{Parametric homology and periodic homotopy}
\label{sec:homologyvshomotopy}

In this section we combine the methods of previous sections to prove most of our main results, namely \Cref{thm:mainabs}, \Cref{thm:mainhomologytohtpy}, \Cref{thm:mainhtpytohomology}, \Cref{thm:Hmainhomologytohtpy}, and \Cref{thm:Hmainhtpytohomology}.

\subsection{The absolute case}

First we will prove \Cref{thm:mainabs} and \Cref{thm:mainhtpytohomology}. The prime $p$ is fixed throughout. We begin with the following characterization of the maps $f$ inducing $T(i)$-equivalences for $i \leq n$, which is due to Bousfield \cite[Section 13]{Bou94}. We will also briefly review the proof for the reader's convenience.

\begin{thm}[Bousfield]
\label{thm:Freudenthal}
    Let $n \geq  0$ be a natural number. For a map of spaces $f \colon X \to Y$ the following are equivalent:
    \begin{enumerate}
        \item 
        The map $f$ is a $T(i)$-equivalence for all $i \leq n$.
        \item
        For any choice of basepoint in $X$, the map $\susp^k f$ is a $v_i$-periodic equivalence for all $i \leq n$ and $k\gg 0$.
        \item 
        For any pointed finite space $V_{n+1}$ of type $n+1$, the map $\mathbf{P}_{\susp V_{n+1}}(\susp^k f_{(p)})$ is an equivalence for $k 
        \gg 0$.
    \end{enumerate}
\end{thm}

To prepare for the proof of \Cref{thm:Freudenthal} we establish the following:

\begin{lem}
\label{lem:inductiveacyclic}
    Let $f\colon X \to Y$ be a map of pointed spaces, $k \geq 0$ a natural number, and $E$ a spectrum. If $\Loop^k f$ is an $E$-equivalence, then the map $\mathrm{Map}_*(V,f)$ is an $E$-equivalence for any finite space $V$ of dimension $\leq k$.
\end{lem}
\begin{proof}
For $k=0$ this is the statement that a finite product of $E$-equivalences is an $E$-equivalence. We proceed by induction and assume we have proved the statement for $k-1$. First observe that the hypothesis of the lemma implies that $\Loop^{\ell} f$ is an $E$-equivalence for every $\ell \leq k$ by applying the bar construction $k - \ell$ times. Fix a finite cell structure on $V$ and consider the cofiber sequence $\mathrm{sk}_{k-1} V \to V \to \bigvee S^k$, with the wedge being over the $k$-cells of $V$. Then we have a principal fiber sequence 
\[
\mathrm{Map}_*(\mathrm{sk}_{k-1} V, \Loop X) \to \prod \Loop^k X \to \mathrm{Map}_*(V,X)
\]
and in particular an associated equivalence
\[
\mathrm{colim}_{\mathbf{\Delta}^{\mathrm{op}}} \bigl(\mathrm{Map}_*(\mathrm{sk}_{k-1} V, \Loop X)^{\times \bullet} \times \prod \Loop^k X\bigr) \xrightarrow{\simeq} \mathrm{Map}_*(V,X).
\]
Hence it suffices to observe that the maps $\prod \Loop^k f$ and $\mathrm{Map}_*(\mathrm{sk}_{k-1} V, \Loop f)$ are $E$-equivalences. The first is true by the hypothesis of the lemma, the second by the inductive hypothesis. 
\end{proof}

\begin{proof}[Proof of \Cref{thm:Freudenthal}]
Without loss of generality we may choose a basepoint in $X$ and take $f$ to be a pointed map. To prove $(1 \Rightarrow 2)$ it suffices to show that $\susp^k f$ is a $v_n$-periodic equivalence for $k \gg 0$; indeed, the same argument applied with $i$ in place of $n$ will then show $\susp^k f$ is also a $v_i$-equivalence. By \Cref{lem:periodic-equivalence-criterion} it suffices to argue that $\susp^\infty \mathrm{Map}_*(V_n, \susp^k f)$ is a $T(n)$-equivalence of spectra for $k \gg 0$. Take $k > \mathrm{dim}(V_n)$. According to Lemma \ref{lem:inductiveacyclic} it then suffices to argue that $\Loop^{\ell}\susp^k f$ is a $T(n)$-equivalence for $\ell = \mathrm{dim}(V_n)$. Snaith splitting implies that the functor $\Loop^{\ell}\susp^{\ell}$ preserves $T(n)$-equivalences of connected spaces; applying this to the map $\susp^{k-\ell} f$ gives the result. The implication $(2\Rightarrow 3)$ follows from \Cref{thm:periodichtpyPvn} and the fact that $\langle W_{n+1} \rangle \geq \langle V_{n+1} \rangle$. Finally, for $(3) \Rightarrow (1)$ we conclude from \Cref{prop:fequivEequiv} and the fact that $V_{n+1}$ is $T(i)$-acyclic for $i\leq n$ that $\susp^k f_{(p)}$ is a $T(i)$-equivalence. This implies that $f$ itself is a $T(i)$-equivalence.
\end{proof}

We are now in position to deduce our first main results. Note that the equivalence between (1) and (2) in the following proves \Cref{thm:mainabs}.

\begin{thm}
\label{thm:parametricallyacyclic}
For a weakly $p$-local pointed space $X$, the following are equivalent:
\begin{enumerate}
\item The space $X$ is parametrically $T(i)$-acyclic for every $i \leq n$.
\item The space $X$ is $(n+1)$-connected and has trivial $v_i$-periodic homotopy groups for every $i \leq n$. 
\item The space $X$ satisfies $\mathbf{P}_{\susp W_{n+1}}(X) \simeq 0$, with $W_{n+1}$ as in \Cref{thm:Pvnperiodichtpy}.
\end{enumerate}
\end{thm}
\begin{proof}
Clearly each of the three conditions implies that $X$ is simply connected (where we rely on \Cref{cor:parametrictau1equiv} in the case of item (1)), so we take $X$ simply-connected throughout the proof. The hypothesis that $X$ is weakly $p$-local then implies it is actually $p$-local. We pick an arbitrary basepoint $x \in X$. 

\Cref{lem:parametricvsordinary} implies that (1) is equivalent to $\Loop X$ being $T(i)$-acyclic for all $i \leq n$. In particular, this implies that $\Loop X$ is $p^\infty$-torsion (since it is rationally acyclic) and $n$-connected (by \Cref{prop:T(n)equivtruncation}). \Cref{cor:bousfield-classification} then gives the inequality
\[
\langle \susp \Loop X \rangle \leq \langle \susp^k \Loop X \rangle \vee \langle K(\mathbb{Z}/p, n+2) \rangle
\]
for $k \geq 1$. Furthermore, \Cref{thm:Freudenthal} shows that (1) is equivalent to $\mathbf{P}_{\susp V_{n+1}}(\susp^k \Loop X) \simeq 0$ for $k \gg 0$, or in other words $\langle \susp V_{n+1} \rangle \geq\langle \susp^k \Loop X \rangle$. This implies
\[
\langle \susp W_{n+1} \rangle = \langle \susp V_{n+1} \rangle \vee \langle \susp K(\mathbb{Z}/p, n+1) \rangle \geq \langle \susp^k \Loop X \rangle \vee \langle K(\mathbb{Z}/p, n+2)\rangle \geq \langle \susp \Loop X\rangle \geq \langle X \rangle
\]
where the first inequality uses \Cref{rmk:LocLoop}, the second inequality was established above, and the third is \Cref{cor:LocLoop}(2). The resulting inequality $\langle \susp W_{n+1} \rangle \geq \langle X \rangle$ is precisely item (3). If $X$ satisfies (3), then it is certainly $(n+1)$-connected (since $\susp W_{n+1}$ is $(n+1)$-connected) and \Cref{thm:Pvnperiodichtpy}(1) shows that it has vanishing $v_i$-periodic homotopy groups for $i \leq n$, showing that (2) holds. Conversely, applying \Cref{thm:periodichtpyPvn} to the map $* \to X$ shows that (2) implies (3). Finally, suppose that $X$ satisfies (3). First observe that $W_{n+1}$ is $T(i)$-acyclic for every $i \leq n$ (see \Cref{prop:T(n)localEMspaces}(1)). Hence \Cref{prop:fequivEequiv} implies that $X \to \mathbf{P}_{\susp W_{n+1}}(X)$ is a parametric $T(i)$-equivalence for $i \leq n$ and (1) follows immediately.
\end{proof}

Our next result now follows rather easily:

\begin{proof}[Proof of \Cref{thm:mainhtpytohomology}]
We take $f\colon X \to Y$ to be an $(n+2)$-connected map of spaces such that $f$ is a $v_i$-periodic equivalence at every basepoint, for all $i \leq n$. Working one component at a time we may suppose that $X$ and $Y$ are connected. After picking a basepoint of $X$, it will suffice to show that the fiber $F$ of $f$ is parametrically $T(i)$-acyclic for every $i \leq n$. But $F$ is $(n+1)$-connected and has trivial $v_i$-periodic homotopy groups for $i \leq n$ by hypothesis, so this follows from \Cref{thm:parametricallyacyclic}. (Note that we may without loss of generality replace $\Loop F$ by its $p$-localization, as this does not affect its $v_i$-periodic homotopy groups or its $T(i)$-homology.)
\end{proof}

\subsection{From parametric $T(n)$-homology to $v_n$-periodic homotopy}

We will proceed to prove \Cref{thm:mainhomologytohtpy}. We rely on the following sharpening of \cite[Theorem 12.5]{Bou94}:

\begin{prop}
\label{prop:parametricvn}
Let $f\colon X \to Y$ be a virtual $T(n)$-equivalence of pointed spaces. Then $f$ is also a $v_n$-periodic equivalence.
\end{prop}
\begin{proof}
By \Cref{thm:virtualT(n)}, the induced map $\tau_{>i} \Loop^k X \to \tau_{>i} \Loop^k Y$ is a $T(n)$-equivalence for every $k \geq 1$ and $i \geq n+2$. If $V$ is a finite pointed space, \Cref{lem:inductiveacyclic} then implies that 
$\mathrm{Map}_*(V,\tau_{>i}X) \to \mathrm{Map}_*(V,\tau_{>i}Y)$ is an equivalence for $i \geq \mathrm{dim}(V)+n+2$. Now let $V$ be of type $n$, suppose it admits a $v_n$ self-map, and consider the commutative diagram
\[
\begin{tikzcd}
\mathrm{Map}_*(V,\tau_{>i}X) \ar{d}{f_*}\ar{r} & \Loop^\infty\Phi_V(\tau_{>i }X) \ar{d}\\
\mathrm{Map}_*(V,\tau_{>i}Y) \ar{r} & \Loop^\infty\Phi_V(\tau_{>i}Y). 
\end{tikzcd}
\]
According to \Cref{lem:basicpropertiesPhi} the horizontal maps are $v_n$-periodic equivalences and the two spaces on the right are $T(n)$-local. Since the left vertical map is a $T(n)$-equivalence, there exists a lift in the square. By two-out-of-six, every map in the square must then be a $v_n$-periodic equivalence. In particular this holds true for the right vertical map. Since $\Phi_V(\tau_{>i }X) \simeq \Phi_V(X)$, and similarly for $Y$, this completes the proof.
\end{proof}

\begin{proof}[Proof of \Cref{thm:mainhomologytohtpy}]
A parametric $T(n)$-equivalence $f$ is in particular a virtual $T(n)$-equivalence by \Cref{lem:parametricimpliesvirtual}. The theorem thus follows by combining \Cref{prop:parametricvn} (showing that $f$ is a $v_n$-periodic equivalence) and \Cref{cor:T(n)equivtruncation} (showing that $\tau_{\leq n+1} f$ is an equivalence).
\end{proof}

\subsection{Consequences for $T(n)$-equivalences}

Finally, we now deal with \Cref{thm:Hmainhomologytohtpy} and \Cref{thm:Hmainhtpytohomology}.

\begin{proof}[Proof of \Cref{thm:Hmainhomologytohtpy}]
A $T(n)$-equivalence $f$ of $H$-spaces is a virtual $T(n)$-equivalence by \Cref{prop:virtualequivHspace} and hence, by \Cref{prop:parametricvn}, a $v_n$-periodic equivalence. The fact that $\tau_{\leq n} f$ is also an equivalence is the content of \Cref{prop:T(n)equivtruncation}.
\end{proof}

\begin{proof}[Proof of \Cref{thm:Hmainhtpytohomology}]
We consider an $(n+1)$-connected map $f\colon X \to Y$ that is a $v_i$-periodic equivalence at every basepoint, for all $i \leq n$. Choose an arbitrary basepoint $x \in X$. Then the fiber $F$ of $f$ at $f(x)$ is an $n$-connected space with vanishing $v_i$-periodic homotopy groups for $i \leq n$. It will suffice to show that $F$ is $T(i)$-acyclic. The $(n+1)$-connected cover $\tau_{> n+1} F$ still has vanishing $v_i$-periodic homotopy groups for $i \leq n$ and is therefore parametrically $T(i)$-acyclic by \Cref{thm:mainhtpytohomology} and in particular $T(i)$-acyclic (see \Cref{lem:paramordinaryEequiv}). It follows that the map $F \to \tau_{\leq n+1} F$ is a $T(i)$-equivalence. But $\tau_{\leq n+1} F$ is a $K(A,n+1)$ for some torsion abelian group $A$ (since $F$ has vanishing rational homotopy groups) and therefore $T(i)$-acyclic by \Cref{prop:T(n)localEMspaces}.
\end{proof}

  
\section{The case of infinite loop spaces}
\label{sec:infiniteloop}

In the case of infinite loop spaces $\Loop^\infty E$ of spectra with $\Loc_{n-1}^f E \simeq 0$ our main results admit a significant sharpening, namely \Cref{thm:maininfloop}. The proof of this result will rely on our previous results, but we also need some extra ideas concerning the nilcompletion of commutative $T(n)$-local ring spectra. 

\Cref{subsec:nilcompletion} starts with a discussion of nilcompletion of commutative ring spectra. We deduce consequences for the $T(n)$-homology of the Whitehead tower of an infinite loop space in \Cref{subsec:Tnghighlyconnected}, which includes a proof of \Cref{cor:KuhnT(n)}. In \Cref{subsec:Tnlocalinfloop} we then prove \Cref{thm:maininfloop} on parametric $T(n)$-equivalences of infinite loop spaces and \Cref{cor:T(n)localinfiniteloop} on the $\Loc_n^f$-localization of infinite loop spaces. Finally, in \Cref{subsec:generalizations} we indicate some generalizations of our arguments from the case of infinite loop spaces to the case of general spaces, which provide a positive answer to Question 1 of the introduction in some special cases.

\subsection{Nilcompletion of $T(n)$-local commutative rings}
\label{subsec:nilcompletion}

We begin by reviewing some general facts on the completion of commutative ring spectra. We will then specialize to the $T(n)$-local context and consider the behavior of completion on two types of rings, namely free $T(n)$-local commutative algebras and the $T(n)$-localized suspension spectra of Eilenberg--MacLane spaces.
 
Let $\mathcal{C}$ be a stable presentably symmetric monoidal $\infty$-category and write $\mathrm{CAlg}^{\mathrm{aug}}(\mathcal{C})$ for the $\infty$-category of augmented commutative algebras in $\mathcal{C}$ (where the augmentation is to the unit $\mathbf{1}$ of $\mathcal{C}$). Then there is an adjoint pair of functors
\[
\begin{tikzcd}
   \mathrm{CAlg}^{\mathrm{aug}}(\mathcal{C}) \ar[shift left]{r}{\mathrm{cot}} & \mathcal{C}, \ar[shift left]{l}{\mathrm{triv}}  
\end{tikzcd}
\]
where the right adjoint sends an object $X$ of $\mathcal{C}$ to the trivial square-zero extension $\mathbf{1} \oplus X$ and the left adjoint takes the \emph{cotangent fiber} of an augmented commutative algebra (also often referred to as \emph{topological Andr\'{e}--Quillen homology}). Moreover, this adjunction exhibits $\mathcal{C}$ as the stabilization of $\mathrm{CAlg}^{\mathrm{aug}}(\mathcal{C})$ \cite[Proposition 6.3]{heutskoszul}. From this perspective, the functor $\mathrm{triv}$ is analogous to $\Loop^\infty$ and the functor $\mathrm{cot}$ to $\susp^\infty$.

In fact, to every $R \in \mathrm{CAlg}^{\mathrm{aug}}(\mathcal{C})$ one can naturally associate a tower of augmented commutative algebras
\[
\begin{tikzcd}
& \vdots \ar{d} \\
& t_3 R \ar{d} \\
& t_2 R \ar{d} \\
R \ar{r}\ar{ur}\ar{uur} & t_1 R = \mathrm{triv}(\mathrm{cot}(R)).    
\end{tikzcd}
\]
This is the nilcompletion tower considered in \cite[Section 9]{heutskoszul} or alternatively the Goodwillie tower of the identity on the $\infty$-category $\mathrm{CAlg}^{\mathrm{aug}}(\mathcal{C})$ (cf. \cite[Section 3]{kuhnAQG}). (This same tower also arises from the adic filtration of commutative algebras considered in \cite[Section 5.1]{brantnermathew}.) We refer to the map $R \to \lim_k t_k R$ as the \emph{nilcompletion} of $R$. This nilcompletion is an analog in higher algebra of the completion of an ordinary augmented commutative ring at its augmentation ideal.

We will need a few basic properties of the nilcompletion tower. In the following, we write 
\[
\mathrm{Sym}(X) = \bigoplus_{k \geq 0} X^{\otimes k}_{h\Sigma_k}
\]
for the free augmented commutative algebra on an object $X \in \mathcal{C}$. The functor $\mathrm{Sym}$ can be constructed as the left adjoint of the augmentation ideal functor
\[
\mathfrak{m}\colon \mathrm{CAlg}^{\mathrm{aug}}(\mathcal{C}) \to \mathcal{C}\colon (R \xto{\varepsilon} \mathbf{1}) \mapsto \mathrm{fib}(\varepsilon).
\]
For the proof of the following proposition, item (1) can be found in \cite[Theorem 9.2]{heutskoszul} or \cite[Theorem 3.10]{kuhnAQG}, whereas item (2) is immediate from the construction of $t_k$ in \cite{heutskoszul} or stated explicitly as \cite[Example 3.14]{kuhnAQG}.

\begin{prop}
\label{prop:nilcompletion}
\begin{enumerate}
\item The layers of the nilcompletion tower are given by
\[
\mathrm{fib}(t_k R \to t_{k-1} R) \cong \mathrm{cot}(R)^{\otimes k}_{h\Sigma_k}.
\]
More succinctly, the associated graded of the nilcompletion tower of $R$ is the free commutative algebra on the cotangent fiber $\mathrm{cot}(R)$.
\item The nilcompletion tower of $\mathrm{Sym}(X)$ can be identified with the system of natural projection maps
\[
\bigoplus_{m = 0}^\infty X^{\otimes m}_{h\Sigma_m} \to \bigoplus_{m \leq k} X^{\otimes m}_{h\Sigma_m}
\]
as $k$ ranges from 1 to $\infty$. In particular, the nilcompletion $\lim_k t_k \mathrm{Sym}(X)$ can be identified with the infinite product $\prod_m X^{\otimes m}_{h\Sigma_m}$.
\end{enumerate}
\end{prop}

We will now specialize to the case $\mathcal{C} = \Sp_{T(n)}$. For a free commutative algebra $\mathrm{Sym}(X)$ we use the shorthand $\widehat{\mathrm{Sym}}(X) := \lim_k t_k \mathrm{Sym}(X)$ for its nilcompletion.

\begin{lem}
\label{lem:nilcompleteSym}
Let $X$ be a $T(n)$-local spectrum. Then the homomorphism
\[
T(n)_* \mathrm{Sym}(X) \to T(n)_* \widehat{\mathrm{Sym}}(X)
\]
induced by nilcompletion is injective.
\end{lem}
\begin{proof}
The spectrum $T(n)$ is the telescope of a $v_n$-self map on a finite spectrum $V$; in fact, $T(n) = L_{T(n)} V$. For a $T(n)$-local spectrum $Y$, we can therefore interpret the groups $T(n)_*Y$ as the homotopy groups of the mapping spectrum $\mathrm{map}(L_{T(n)} V^{\vee}, Y)$. Since $L_{T(n)}V^{\vee}$ is a compact object of $\Sp_{T(n)}$ and $\mathrm{Sym}(X)$ is an infinite direct sum of terms $X^{\otimes k}_{h\Sigma_k}$ in $\Sp_{T(n)}$, any map $f\colon L_{T(n)}V^{\vee} \to \mathrm{Sym}(X)$ factors through a finite sum. In particular, if the composition of $f$ with the map $\mathrm{Sym}(X) \to \widehat{\mathrm{Sym}}(X)$ to the infinite product is null, then $f$ itself must have been null.
\end{proof}


The preceding lemma shows that the nilcompletion of a free commutative algebra does not really lose any information. Quite the opposite is true for commutative algebras of the following kind:

\begin{lem}
\label{lem:nilcompleteEM}
Let $A$ be a torsion abelian group and consider the augmented commutative algebra $R := L_{T(n)}\susp^\infty_+K(A,n) \in \mathrm{CAlg}^{\mathrm{aug}}(\Sp_{T(n)})$. Then $\mathrm{cot}(R) \cong 0$ and hence the augmentation $R \to \mathbb{S}_{T(n)}$ exhibits the unit as the nilcompletion of $R$.
\end{lem}
\begin{proof}
Since the functor $\mathrm{cot}$ preserves colimits and its codomain is stable, it suffices to argue that the suspension of $R$ is trivial (i.e., equivalent to the zero object $\mathbb{S}_{T(n)}$) in the $\infty$-category $\mathrm{CAlg}^{\mathrm{aug}}(\Sp_{T(n)})$. This suspension is precisely the bar construction
\[
\mathbb{S}_{T(n)} \otimes_R \mathbb{S}_{T(n)} \cong L_{T(n)}\susp^\infty_+ BK(A,n) = L_{T(n)}\susp^\infty_+ K(A,n+1).
\]
The space $K(A,n+1)$ is $T(n)$-acyclic by \Cref{prop:T(n)localEMspaces}, giving the conclusion.
\end{proof}

\begin{rem}
\label{rem:E1indec}
In fact the proof of \Cref{lem:nilcompleteEM} shows something stronger: the $\mathbb{E}_1$-cotangent fiber of the ring $R$ already vanishes. Indeed, up to a shift, the $\mathbb{E}_1$-cotangent fiber is precisely the bar construction \cite[Theorem 1.3]{basterramandell}. We will come back to this point in \Cref{subsec:generalizations}.
\end{rem}

\subsection{The $T(n)$-homology of highly connected covers}
\label{subsec:Tnghighlyconnected}

In the introduction we stated \Cref{cor:KuhnT(n)} as a consequence of \Cref{thm:maininfloop}. In fact, our proof will proceed by establishing the corollary first; we state it here for the reader's convenience.

\begin{prop}
\label{prop:Tnconnectivecover}
Suppose $E$ is a spectrum with $\Loc^f_{n-1} E = 0$. Then for any $m > n$ the map $\tau_{>m} \Loop^\infty E \to \tau_{>n} \Loop^\infty E$ is a $T(n)$-equivalence.
\end{prop}

Let $A$ be a finite group. Carmeli--Schlank--Yanovski \cite[Section 4.2]{CSYcyclotomic} show that $T(n)$-local commutative rings of the kind  $L_{T(n)}\susp^{\infty}_+ K(A,n)$ admit interesting idempotents, splitting them into factors. Moreover, one of these factors is the unit $\mathbb{S}_{T(n)}$. The precise statement we need is the following:

\begin{prop}
\label{prop:ambidexterityidempotent}
Let $A$ be a finite group and $R := L_{T(n)}\susp^{\infty}_+ K(A,n)$. Then there exists an idempotent $e \in \pi_0 R$ so that the augmentation $R \to \mathbb{S}_{T(n)}$ sends $e$ to 1 and the resulting map $R[e^{-1}] \to \mathbb{S}_{T(n)}$ is an equivalence.
\end{prop}
\begin{proof}
Start with the standard identification of stable $\infty$-categories
\[
\mathrm{LMod}_R(\Sp_{T(n)}) \simeq \mathrm{Fun}(K(A,n+1), \Sp_{T(n)}).
\]
Write $t\colon K(A,n+1) \to *$ for the unique such map. Then there is a triple of adjoint functors
\[
\begin{tikzcd}
\mathrm{Fun}(K(A,n+1), \Sp_{T(n)}) \ar[shift left = 2ex]{r}{t_!} \ar[shift right = 2ex]{r}[swap]{t_*} & \Sp_{T(n)}, \ar[l, "{t^*}"{description, fill=white, inner sep=1pt}]
\end{tikzcd}
\]
with left adjoints written above the corresponding right adjoints. The functor $t_!$ computes the colimit over $K(A,n+1)$, the functor $t_*$ computes the limit. We claim that $t^*$ is fully faithful. Indeed, for any $X \in \Sp_{T(n)}$ the counit map $t_!t^*(X) \to X$ is an equivalence, as can be seen by identifying this map with $X \otimes K(A,n+1)_+ \to X$ and observing that the space $K(A,n+1)$ is $T(n)$-acyclic by \Cref{prop:T(n)localEMspaces}. It now follows that the three adjoints above form part of a recollement (in the sense of \cite[Section A.8.1]{HA}), assembling the $\infty$-category $\mathrm{Fun}(K(A,n+1), \Sp_{T(n)})$ out of $\Sp_{T(n)}$ and its right orthogonal complement $\mathrm{ker}(t_*)$. Ambidexterity in the $T(n)$-local category provides an equivalence $t_! \cong t_*$, forcing this recollement to be \emph{split} (see \cite[Proposition 4.1.4]{CSYheight}): that is, we find an equivalence of $\infty$-categories
\[
\mathrm{Fun}(K(A,n+1), \Sp_{T(n)}) \simeq \Sp_{T(n)} \times \mathrm{ker}(t_*).
\]
This splitting of the category induces a splitting of the unit object of $\mathrm{Fun}(K(A,n+1), \Sp_{T(n)})$. The endomorphism ring of this unit is precisely $R$; we take $e$ to be the idempotent corresponding to the factor $\Sp_{T(n)}$. The desired properties are clear.
\end{proof}

The following is the crux of our proof of \Cref{prop:Tnconnectivecover}:

\begin{prop}
\label{prop:mapKpiSym}
Let $A$ be a torsion abelian group, write $R := \Loc_{T(n)} \susp^\infty_+ K(A,n)$, and let $X$ be a $T(n)$-local spectrum. Then the mapping space $\mathrm{Map}_{\mathrm{CAlg}^{\mathrm{aug}}(\Sp_{T(n)})}(R, \mathrm{Sym}(X))$ is contractible. In particular, every map $R \to \mathrm{Sym}(X)$ factors canonically over the augmentation $R \to \mathbb{S}_{T(n)}$.
\end{prop}
\begin{proof}
Writing $A$ as a filtered colimit of finite groups, we reduce to the case where $A$ is finite. Consider a map of augmented commutative algebras $f\colon R \to \mathrm{Sym}(X)$. Together with the nilcompletions of these rings, this yields the following commutative square:
\[
\begin{tikzcd}
R \ar{r}{f}\ar{d} & \mathrm{Sym}(X) \ar{d} \\
\mathbb{S}_{T(n)} \ar{r} & \widehat{\mathrm{Sym}}(X).
\end{tikzcd} 
\]
Here we have applied \Cref{lem:nilcompleteEM} to identify the nilcompletion of $R$. Without loss of generality we can arrange that the telescope $T(n)$ is a ring spectrum (e.g. by inverting a $v_n$ self-map on the endomorphism spectrum of a type $n$ complex). Write $\overline{e} \in T(n)_0R$ for the image of the idempotent $e \in \pi_0 R$ under the Hurewicz map. It follows from \Cref{prop:ambidexterityidempotent} and the square above that the image of $1-\overline{e}$ in $T(n)_0\widehat{\mathrm{Sym}}(X)$ is 0. By \Cref{lem:nilcompleteSym}, the image of $1-\overline{e}$ in $T(n)_0\mathrm{Sym}(X)$ is therefore 0 as well. It follows that the localization $\mathrm{Sym}(X)[(1-f(e))^{-1}]$ vanishes. Since $1-f(e)$ is idempotent, we must have $1-f(e) = 0$. We conclude that $f$ factors over the localization $R \to R[e^{-1}]$. The proof is completed by applying \Cref{prop:ambidexterityidempotent}.
\end{proof}


\begin{proof}[Proof of \Cref{prop:Tnconnectivecover}]
We begin with the critical case $m = n+1$. Our aim is to show that for a spectrum $E$ with vanishing $v_i$-periodic homotopy groups for $i < n$, the map $\tau_{>n+1} \Loop^\infty E \to \tau_{>n} \Loop^\infty E$ is a $T(n)$-equivalence. Consider the principal fiber sequence
\[
K(\pi_{n+1} E, n) \to \tau_{>n+1} \Loop^\infty E \to \tau_{>n} \Loop^\infty E.
\]
Set $R := \Loc_{T(n)} \susp^\infty_+ K(\pi_{n+1} E, n)$. We note that $\pi_{n+1} E$ is torsion, since we have assumed the rational homotopy groups of $E$ to vanish. The two maps above are infinite loop maps, so in particular we obtain a morphism of commutative $T(n)$-local rings $R \to \Loc_{T(n)}\susp^\infty_+ \tau_{>n+1} \Loop^\infty E$. Furthermore, the fiber sequence gives an equivalence
\[
\mathbb{S}_{T(n)} \otimes_{R} \Loc_{T(n)}\susp^\infty_+ \tau_{>n+1} \Loop^\infty E \xrightarrow{\cong}  \Loc_{T(n)} \susp^\infty_+ \tau_{>n} \Loop^\infty E.
\]
A result of Kuhn \cite[Theorem 2.10, Theorem 2.12]{kuhnAQG} implies that there is an equivalence of $T(n)$-local commutative rings
\[
\Loc_{T(n)}\susp^\infty_+ \tau_{>n+1} \Loop^\infty E \cong \mathrm{Sym}(\Loc_{T(n)} E).
\]
\Cref{prop:mapKpiSym} therefore shows that the ring map $R \to \Loc_{T(n)}\susp^\infty_+ \tau_{>n+1} \Loop^\infty E$ factors over the augmentation and we obtain equivalences
\[
\mathbb{S}_{T(n)} \otimes_{R} \Loc_{T(n)}\susp^\infty_+ \tau_{>n+1} \Loop^\infty E \cong (\mathbb{S}_{T(n)} \otimes_{R} \mathbb{S}_{T(n)}) \otimes \Loc_{T(n)}\susp^\infty_+ \tau_{>n+1} \Loop^\infty E \cong \Loc_{T(n)}\susp^\infty_+ \tau_{>n+1} \Loop^\infty E.
\]
Here we have used $\mathbb{S}_{T(n)} \otimes_R \mathbb{S}_{T(n)} \cong \mathbb{S}_{T(n)}$, which we already showed in the proof of \Cref{lem:nilcompleteEM}. It follows that the map $\Loc_{T(n)}\susp^\infty_+ \tau_{>n+1} \Loop^\infty E \to \Loc_{T(n)}\susp^\infty_+ \tau_{>n} \Loop^\infty E$ is an equivalence, completing the proof of the case $m = n+1$.

The cases $m > n+1$ can now be handled by an easy inductive argument. Indeed, consider the map $\tau_{>m} \Loop^\infty E \to \tau_{>m-1} \Loop^\infty E$. Its fiber is $K(\pi_m E, m-1)$. Since $m - 1 > n$ and $\pi_m E$ is torsion, this space is $T(n)$-acyclic by \Cref{prop:T(n)localEMspaces} and the map is a $T(n)$-equivalence.
\end{proof}

\subsection{On the $T(n)$-localization of infinite loop spaces}
\label{subsec:Tnlocalinfloop}

In this section we prove \Cref{thm:maininfloop} and \Cref{cor:T(n)localinfiniteloop}.

\begin{proof}[Proof of \Cref{thm:maininfloop}]

Let $E$ and $F$ be $p$-local spectra with $\Loc_{n-1}^f E \simeq 0 \simeq \Loc_{n-1}^f F$ and let $f\colon \Loop^\infty E \to \Loop^\infty F$ be a map of pointed spaces (but not necessarily a loop map). If $f$ is a $T(n)$-equivalence, then \Cref{thm:Hmainhomologytohtpy} implies that $\tau_{\leq n} f$ is an equivalence and $f$ is a $v_n$-periodic equivalence. If $f$ is moreover a parametric $T(n)$-equivalence, then \Cref{thm:mainhomologytohtpy} implies that also $\tau_{\leq n + 1} f$ is an equivalence.

Conversely, assume that $\tau_{\leq n} f$ is an equivalence and $f$ is a $v_n$-periodic equivalence at any basepoint. Consider the diagram of fiber sequences
\[
\begin{tikzcd}
\tau_{>n} \Loop^\infty E \ar{r}\ar{d}{\tau_{>n} f} & \Loop^\infty E \ar{r}\ar{d}{f} & \tau_{\leq n}\Loop^\infty E \ar{d}{\tau_{\leq n} f}[swap]{\simeq} \\
\tau_{>n} \Loop^\infty F \ar{r} & \Loop^\infty F \ar{r} & \tau_{\leq n}\Loop^\infty F.
\end{tikzcd}
\]
To show that $f$ is a $T(n)$-equivalence, it suffices to show that the map of fibers $\tau_{>n}f$ is a $T(n)$-equivalence. Applying \Cref{prop:Tnconnectivecover}, it will in fact suffice to show that $\tau_{>n+1} f$ is a $T(n)$-equivalence. Note that this is an $n$-connected map that is still a $v_n$-periodic equivalence. Moreover, the $v_i$-periodic homotopy groups of $\Loop^\infty E$ and $\Loop^\infty F$ vanish for $i < n$ by hypothesis. Hence \Cref{thm:Hmainhtpytohomology} applies and the desired conclusion follows.

Finally, suppose that $\tau_{\leq n+1} f$ is an equivalence and $f$ is a $v_n$-periodic equivalence at any basepoint. To show that $f$ is a parametric $T(n)$-equivalence, it suffices to work one component at a time and reduce to the case where $E$ and $F$ are connected. Then it suffices to show that $\Loop f$ is a $T(n)$-equivalence; since $\tau_{\leq n} \Loop f = \Loop \tau_{\leq n+1} f$, this follows from our previous argument.
\end{proof}

\begin{proof}[Proof of \Cref{cor:T(n)localinfiniteloop}]
Let $E$ be a $p$-local spectrum with $\Loc_{n-1}^f E \cong 0$ and define an infinite loop space $X$ by the pullback square
\[
\begin{tikzcd}
X \ar{r}\ar{d} & \Loop^\infty \Loc_n^f E \ar{d} \\
\tau_{\leq n}\Loop^\infty E \ar{r} & \tau_{\leq n}\Loop^\infty \Loc_n^f E.
\end{tikzcd}
\]
Then there is an evident infinite loop map $f\colon \Loop^\infty E \to X$. The corollary follows if we prove that $X$ is an $\Loc_n^f$-local space and that $f$ is an $\Loc_n^f$-equivalence, i.e., a $T(i)$-equivalence for every $i \leq n$. First observe that $\Loop^\infty \Loc_n^f E$ is $\Loc_n^f$-local. Also, the two spaces on the bottom row are $\Loc_n^f$-local by \cite[Theorem 7.4]{Bou01}. (Alternatively, this can also be deduced from \Cref{prop:T(n)localEMspaces}.) It follows that $X$ is indeed $\Loc_n^f$-local.

We will show that $f$ is a $T(i)$-equivalence for each $i \leq n$ by checking that it satisfies condition (2) of \Cref{thm:maininfloop}. First note that $f$ induces an equivalence on $n$-truncations by construction. To see that $f$ is a $v_i$-periodic equivalence for $i \leq n$, consider the maps $\Loop^\infty E \to X \to \Loop^\infty \Loc_n^f E$. The composite is a $v_i$-periodic equivalence by construction, whereas the second map is a $v_i$-periodic equivalence by inspection of the pullback square defining $X$. Indeed, the bottom map is clearly a rational equivalence and also a $v_i$-periodic equivalence for $i \geq 1$ simply because both objects are truncated, hence $v_i$-periodically trivial. By two-out-of-three we conclude that the initial map $\Loop^\infty E \to X$ must be a $v_i$-periodic equivalence as well, for every $i \leq n$, completing the argument. 
\end{proof}

\subsection{Generalizations}
\label{subsec:generalizations}

In the introduction we raised the following question. Let $f\colon X \to Y$ be a map of spaces with vanishing $v_i$-periodic homotopy groups for $i < n$. If $\tau_{\leq n+1} f$ is an equivalence and $f$ is a $v_n$-periodic equivalence, then does it follow that $f$ is a parametric $T(n)$-equivalence? In this section we indicate how some of the techniques we have used for infinite loop spaces might be adapted to approach this.

In fact, the proof of \Cref{thm:maininfloop} would go through essentially unchanged if we could show that for a pointed space $X$ with vanishing $v_i$-periodic homotopy groups for $i < n$, the map $\tau_{>n+1} \Loop X \to \tau_{>n} \Loop X$ is a $T(n)$-equivalence (cf. \Cref{prop:Tnconnectivecover}). Consider the fiber sequence of loop maps
\[
K(\pi_{n+2} X, n) \to \tau_{>n+1} \Loop X \to \tau_{>n} \Loop X.
\]
This sequence in particular yields a map of augmented $T(n)$-local $\mathbb{E}_1$-rings 
\[
R := \Loc_{T(n)} \susp^\infty_+ K(\pi_{n+2} X, n) \xrightarrow{f} \Loc_{T(n)} \susp^\infty_+ \tau_{>n+1} \Loop X
\]
and an equivalence
\[
\mathbb{S}_{T(n)} \otimes_R \Loc_{T(n)} \susp^\infty_+ \tau_{>n+1} \Loop X \xrightarrow{\simeq} \Loc_{T(n)} \susp^\infty_+ \tau_{>n} \Loop X.
\]
If $f$ factors over the augmentation $R \to \Loc_{T(n)}\mathbb{S}$, then the argument in the proof of \Cref{prop:Tnconnectivecover} goes through to show that $\tau_{>n+1} \Loop X \to \tau_{>n} \Loop X$ is a $T(n)$-equivalence.

The $\infty$-category $\mathrm{Alg}^{\mathrm{aug}}(\Sp_{T(n)})$ of augmented $\mathbb{E}_1$-algebras also admits a nilcompletion tower (see \cite[Section 9]{heutskoszul}): for any augmented $\mathbb{E}_1$-ring $R$ there is a tower $\{t_k X\}_{k \geq 1}$ under $R$ of which the associated graded is the free $\mathbb{E}_1$-algebra on the $\mathbb{E}_1$-cotangent fiber (or $\mathbb{E}_1$-indecomposables) $\mathrm{cot}_{\mathbb{E}_1}(R)$ of $R$. According to \cite[Remark 2.14]{heutsland}, this object agrees up to a shift with the bar construction: $\mathrm{cot}_{\mathbb{E}_1}(R) \simeq \susp^{-1} \mathrm{Bar}(R)$. In particular, we obtain the following analog of \Cref{lem:nilcompleteEM}:

\begin{lem}
    Let $A$ be a torsion abelian group and consider the augmented $T(n)$-local $\mathbb{E}_1$-algebra $R := \Loc_{T(n)}\susp^\infty_+ K(A,n)$. Then the augmentation $R \to \mathbb{S}_{T(n)}$ exhibits the unit as the nilcompletion (in the sense of $\mathbb{E}_1$-algebras) of $R$. 
\end{lem}

To imitate the argument we gave in the case of infinite loop spaces, we would now need to know that the $T(n)$-homology of $\susp^{\infty}_+\tau_{>n+1} \Loop X$ injects into the $T(n)$-homology of its nilcompletion. For an infinite loop space $X$, we identified the $T(n)$-localization of this suspension spectrum with a free commutative algebra. In the case of a general loop space the situation is not so straightforward, but a useful observation is that the result only depends on the $v_n$-periodic homotopy type of $X$. In fact, in \Cref{prop:suspspectrumunivenv} we will identify $\susp^{\infty}_+\tau_{>n+1} \Loop X$ with the universal enveloping algebra of a spectral Lie algebra naturally associated with $X$.

In the following we use that the values of the Bousfield--Kuhn functor $\Phi_n(X)$ naturally come equipped with the structure of a spectral Lie algebra \cite{heutslie}. Moreover, the stabilization of the $\infty$-category of spectral Lie algebras yields a left adjoint functor that we denote
\[
\mathrm{CE}\colon \mathrm{Lie}(\Sp_{T(n)}) \to \Sp_{T(n)},
\]
since it is a version of the Chevalley--Eilenberg complex appropriate to this context. We write $\mathrm{CE}_+$ for the functor $\mathbb{S}_{T(n)} \oplus \mathrm{CE}$. The relevance of this functor is that (analogously to $\susp^\infty_+$) it admits a symmetric monoidal structure with respect to the cartesian product on $\mathrm{Lie}(\Sp_{T(n)})$ and the smash product on $\Sp_{T(n)}$. In particular, it sends the group object $\Phi_n(\Loop X)$ in $\mathrm{Lie}(\Sp_{T(n)})$ to an $\mathbb{E}_1$-algebra in $\Sp_{T(n)}$.

\begin{prop}
\label{prop:suspspectrumunivenv}
    Let $X$ be a pointed space with vanishing $v_i$-periodic homotopy groups for $i < n$. Then there is a natural equivalence of $\mathbb{E}_1$-algebras
    \[
    \Loc_{T(n)} \susp^\infty_+ \tau_{>n+1} \Loop X \simeq \mathrm{CE}_+(\Phi_n(\Loop X)).
    \]
\end{prop}
\begin{proof}
This statement is a fairly direct consequence of the work in \cite{heutslie}, but not explicitly stated there. Let us recall the relevant ingredients. Fix a finite type $n+1$ suspension $V_{n+1}$. Let $d$ be the smallest natural number so that $\pi_d V_{n+1} \neq 0$ and write $\mathcal{L}_n^f$ for the full subcategory of $\mathbf{P}_{V_{n+1}}\mathcal{S}_*$ consisting of $d$-connected spaces. (By \cite[Theorem 3.7]{heutslie}, the $\infty$-category $\mathcal{L}_n^f$ is precisely the localization of the $\infty$-category of $d$-connected spaces at the maps inducing $v_i$-periodic equivalences for $i \leq n$.) Write $\mathcal{M}_n^f$ for the full subcategory of $\mathcal{L}_n^f$ on spaces with trivial $v_i$-periodic homotopy groups for $i < n$. (According to \cite[Section 3.2]{heutslie}, the $\infty$-category $\mathcal{M}_n^f$ is equivalent to $\mathcal{S}_{v_n}$.) By \cite[Proposition 3.19]{heutslie}, the functor $\Loc_{T(n)}\susp^\infty \colon \mathcal{M}_n^f \to \Sp_{T(n)}$ exhibits $\Sp_{T(n)}$ as the stabilization of $\mathcal{M}_n^f$. \cite[Theorem 2.6]{heutslie} shows that the Bousfield--Kuhn functor gives an equivalence of $\infty$-categories $\Phi_n\colon \mathcal{M}_n^f \to \mathrm{Lie}(\Sp_{T(n)})$. The Chevalley--Eilenberg functor was defined as the stabilization of $\mathrm{Lie}(\Sp_{T(n)})$, so that we arrive at a commutative diagram
\[
\begin{tikzcd}
    \mathcal{M}_n^f \ar{r}{\Loc_{T(n)}\susp^\infty_+}\ar{d}{\Phi_n}[swap]{\simeq} & \Sp_{T(n)}. \\
    \mathrm{Lie}(\Sp_{T(n)}) \ar{ur}[swap]{\mathrm{CE}_+} &
\end{tikzcd}
\]
Note that we have added basepoints both on $\susp^\infty$ and $\mathrm{CE}$ here. Now, let $X$ be as in the statement of the proposition. Then the map $\tau_{>d} \Loop X \to \tau_{>n+1} \Loop X$ is a $T(n)$-equivalence by the same straightforward induction on the Postnikov tower used in the proof of \Cref{prop:Tnconnectivecover}. The space $\mathbf{P}_{V_{n+1}} \tau_{>d}\Loop X$ lies in the subcategory $\mathcal{M}_n^f$ (where we use \cite[Proposition 3.6]{heutslie} to see that this space is still $d$-connected). We conclude that there is a string of equivalences
\[
\Loc_{T(n)}\susp^\infty_+ \tau_{>d} \Loop X \simeq \Loc_{T(n)}\susp^\infty_+ \mathbf{P}_{V_{n+1}}\tau_{>d} \Loop X \simeq \mathrm{CE}_+(\Phi_n(\Loop X)),
\]
where we have used the commutativity of the preceding diagram for the second equivalence and \Cref{prop:fequivEequiv}(1) for the first equivalence.
\end{proof}
\begin{rem}
According to \cite[Theorem 1.8]{ACBH}, the $\mathbb{E}_1$-algebra $\mathrm{CE}_+(\Phi_n(\Loop X))$ can also be interpreted as the universal enveloping algebra of the spectral Lie algebra $\Phi_n(X)$.
\end{rem}

Putting the ingredients together, we get the following extension of \Cref{prop:Tnconnectivecover}:

\begin{prop}
\label{prop:T(n)loopspaceconnectivecover}
    Let $X$ be a pointed space with vanishing $v_i$-periodic homotopy groups for $i < n$. If the nilcompletion map (in the sense of $\mathbb{E}_1$-algebras) of $\mathrm{CE}_+(\Phi_n(\Loop X))$ induces an injection on $T(n)$-homology, then the map $\tau_{>n+1}\Loop X \to \tau_{>n}\Loop X$ is a $T(n)$-equivalence of spaces.
\end{prop}

To conclude, let us briefly point out two examples where \Cref{prop:T(n)loopspaceconnectivecover} applies. If $X = \Loop^\infty E$ is an infinite loop space, then the $\mathbb{E}_1$-nilcompletion and $\mathbb{E}_\infty$-nilcompletion of $\Loc_{T(n)}\susp^\infty_+ \tau_{>n+1} X \simeq \mathrm{Sym}(L_{T(n)} E)$ are related as in the following diagram:
\[
\mathrm{Sym}(\Loc_{T(n)}E) \to \mathrm{Sym}(\Loc_{T(n)}E)^{\wedge}_{\mathbb{E}_1} \to \mathrm{Sym}(\Loc_{T(n)}E)^{\wedge}_{\mathbb{E}_\infty}.
\]
The first map is the $\mathbb{E}_1$-nilcompletion, the composite of the two is the $\mathbb{E}_\infty$-nilcompletion. The composite is therefore injective on $T(n)$-homology, as observed in the proof of \Cref{prop:Tnconnectivecover}, and hence the first map is injective on $T(n)$-homology as well. A second example is the case where $X$ is a pointed space for which spectral Lie algebra $\Phi_n(X)$ is free. (This happens, for example, when there exists a type $n$ complex $V_n$ with $v_n$ self-map so that a highly connected cover of $X$ is $T(n)$-equivalent to $\susp^2 V$, cf. \cite[Theorem 2.13]{heutslie}.) In that case its universal enveloping algebra is a free $T(n)$-local $\mathbb{E}_1$-algebra:
\[
\mathrm{CE}_+(\Phi_n(\Loop X)) \cong \Loc_{T(n)} \bigoplus_{k \geq 0} (\susp^{-1}\mathrm{CE}(\Phi_n(X)))^{\otimes k}.
\]
The $\mathbb{E}_1$-nilcompletion of this algebra is the natural map from direct sum to product, which is injective on $T(n)$-homology as in the proof of \Cref{lem:nilcompleteSym}.





    \printbibliography[heading=bibintoc]

\end{document}